\documentclass[reqno]{amsart}
\usepackage{hyperref}
\usepackage{amssymb,amsmath,amsthm, enumerate}
\usepackage{xcolor}
\usepackage[pdftex]{graphics}
\newcommand{\R}{\mathbb{R}}
\newcommand{\Z}{\mathbb{Z}}
\newcommand{\N}{\mathbb{N}}

\def\R{{\mathbb{R}}}

\def\N{{\mathbb{N}}}
\def\Z{{\mathbb{Z}}}
\def\F{{\mathbb{F}}}

\providecommand{\keywords}[1]
{
  \small	
  \textbf{\textit{Keywords---}} #1
}

\begin{document}

\title[\hfilneg Multiplicity results for  fractional $p-$Laplacian ]{Multiplicity results for Schrödinger type fractional $p$-Laplacian boundary value problems}



\author{Emer Lopera}
\address{Universidad Nacional de Colombia-Manizales, Colombia}
\curraddr{}
\email{edloperar@unal.edu.co}

\author{Leandro Rec\^{o}va}
\address{California State Polytechnic University, Pomona - USA.}
\curraddr{}
\email{llrecova@cpp.edu}

\author{Adolfo Rumbos}
\address{ Pomona College, Claremont, California - USA.}
\curraddr{610 N. College Avenue, 91711 Claremont, CA}
\email{arumbos@pomona.edu}
\keywords{Mountain Pass Theorem, Morse Theory, Critical Groups, Comparison Principle}

\subjclass[2010]{Primary 35J20 }

\date{}

\dedicatory{}


\begin{abstract}
  In this work, we study the existence and multiplicity of solutions for the following problem
\begin{equation}\label{probaa1}
   \left\{ \begin{aligned}
    -(\Delta)_{p}^{s} u  + V(x)|u|^{p-2}u &= \lambda f(u),&x\in\Omega;\\
    u&=0,&x\in \R^{N}\backslash\Omega,
    \end{aligned}
    \right.
\end{equation}
where $\Omega\subset\R^{N}$ is an open bounded set with Lipschitz boundary $\partial\Omega$, $N\geqslant 2,$ $V\in L^{\infty}(\R^{N})$, and $(-\Delta)_p^s$ denotes the fractional $p$-Laplacian with $s\in(0,1), 1<p$, $sp<N$, $\lambda>0$, and $f:\R\rightarrow\R$ is a continuous function. We extend the results of Lopera {\it et al.} in \cite{Lopera1} by  proving the existence of a second weak solution for problem (\ref{probaa1}). We apply a variant of the mountain-pass theorem due to Hofer \cite{Hofer2} and infinite-dimensional Morse theory to obtain the existence of at least two solutions.  
\end{abstract}
\numberwithin{equation}{section}
\newtheorem{theorem}{Theorem}[section]
\newtheorem{lemma}[theorem]{Lemma}
\newtheorem{definition}[theorem]{Definition}
\newtheorem{proposition}[theorem]{Proposition}
\newtheorem{prop}[theorem]{Proposition}
\newtheorem{corollary}[theorem]{Corollary}
\newtheorem{remark}[theorem]{Remark}
\allowdisplaybreaks
\maketitle

\section{Introduction}

Let $\Omega$ be an open bounded set in $\R^{N}$, $N\geqslant 2$, with Lipschitz boundary $\partial\Omega$. In this work, we study the existence and multiplicity of solutions for the problem 
\begin{equation}\label{proba1}
   \left\{ \begin{aligned}
    -(\Delta)_{p}^{s} u + V(x)|u|^{p-2}u &= \lambda f(u),&x\in\Omega;\\
    u&=0,&x\in \R^{N}\backslash\Omega,
    \end{aligned}
    \right.
\end{equation}
where $V\in L^{\infty}(\R^{N})$, $f:\R\rightarrow\R$ is a continuous function and $(-\Delta)_{p}^{s}$ denotes the fractional $p$-Laplacian defined by
\begin{equation}(-\Delta)_{p}^{s}u=2\lim_{\varepsilon\rightarrow 0^{+}}\int_{|x-y|>\varepsilon}\frac{|u(x)-u(y)|^{p-2}(u(x)-u(y))}{|x-y|^{N+sp}}\,dy,
\label{plap}
\end{equation} 
with $s\in(0,1),$ $1<p$, $sp<N$, and $\lambda >0$. 

    As pointed it out by Lindgren and Lidqvist in \cite[page $801$]{LindLind1}, it is not sufficient to prescribe the boundary values only on $\partial\Omega$, but instead, we have to assume that $u=0$ in the whole complement $\R^N\backslash\Omega$ because a change in $u$ done outside $\Omega$ can impact the fractional $p$-Laplacian operator $(-\Delta)_{p}^{s}$. For more details, see Nezza {\it et. al} \cite{Nezza1}, Lindgren {\it et al.} \cite{LindLind1}, and references therein.

In this work, the functions $f$ and $V$ will satisfy the following hypotheses:
\begin{enumerate}
    \item [($H_1$)] Assume that $p-1<q<\displaystyle p_{s}^{*}-1,$ where $p_s^{*}:=\frac{Np}{N-sp}$ is the fractional critical Sobolev exponent, and there exist $A,B>0$ such that 
    \begin{eqnarray}
        A(s^q-1)\leqslant f(s)\leqslant B(s^{q}+1),&\quad\mbox{ for } s>0,\label{e11}\\
        f(s)=0, &\quad\mbox{ for }s\leqslant -1. \label{e12}
        \end{eqnarray}

    \item [($H_2$)] There exist $\theta > p$ and $K\in\R$ such that $f$ satisfies the following Ambrosetti-Rabinowitz type condition
    \begin{equation} sf(s)\geqslant \theta F(s)+K,\quad\mbox{ for all }s\in\R,
        \label{ambcond}
    \end{equation}
    where $F(s)=\int_{0}^{s}f(\xi)\,d\xi$, for $s\in\R$,  is the primitive of $f$. 
    \item [($H_3$)]  $V\in L^{\infty} (\mathbb{R}^N)$ and $V(x)\geqslant -c_V$, for a.e. $x\in \mathbb{R}^N$, where $0<c_V<\lambda_1$ and  $\lambda_1$ is the first eigenvalue of $((-\Delta)_p^s , W_0^{s,p}(\Omega))$. 
    
\end{enumerate}

The first two results establish the existence and multiplicity of solutions for problem (\ref{proba1}) when $f(0)\ne 0.$

\begin{theorem} 
   Assume that $\Omega$ is a bounded domain with a Lipschitz boundary $\partial\Omega$ and the hypotheses $(H1)$-$(H3)$ are satisfied with $f(0)\ne 0$. Then, there exists $\lambda_{0}>0$ such that, for all $\lambda\in (0,\lambda_{0})$, problem (\ref{proba1}) has at least two  solutions. 
    \label{maintheo}
\end{theorem}    

To obtain a positive solution, we need to assume that $p\geqslant 2$ to get enough regularity of solutions up to the boundary of $\Omega$ and $V(x)\geqslant 0$, for a.e $x\in\Omega$. In this case, we obtain the following multiplicity result.

\begin{theorem}\label{maintheo2}
In addition to the hypotheses of Theorem \ref{maintheo}, assume that $V(x)\geqslant 0$ for a.e $x\in\Omega$, $p\geqslant 2$, $\Omega$ is bounded and satisfies the interior ball condition at any $x\in \partial \Omega$, and 
$$p-1<q<\min\left\{\displaystyle\frac{sp}{N}p_s^*,  p_s^*-1\right\}.$$  
Then, there exists $\lambda^* >0$ such that, for all $0<\lambda < \lambda^*$, problem \eqref{proba1} has at least two solutions. Moreover:
\begin{itemize}
    \item[(a)] If $f(0)>0$, then both solutions are positive.
    \item[(b)] If $f(0)<0$, then at least one of the solutions is positive.
\end{itemize}
\end{theorem}

\begin{remark}
    {\rm
     Observe that statement (b) encompasses the semipositone case. See, for example, Castro {\it et al.} \cite{CastroFigLop} and references therein.
    }
\end{remark}

For the case in which the function $u\equiv 0$  is a solution of problem (\ref{proba1}), called the trivial solution, to obtain a multiplicity result in this case, we need an additional condition on the primitive of $f$.

\begin{theorem}
    Assume that $\Omega$ is a bounded domain with a Lipschitz boundary $\partial\Omega$ and the hypotheses $(H1)$-$(H3)$ are satisfied. Moreover, assume that $f(0)=0$ and
     \begin{equation}
        \limsup_{s\rightarrow 0} \frac{F(s)}{|s|^{p}}=0.
        \nonumber 
    \end{equation}
    Then, there exists $\lambda_{0}>0$ such that, for all $\lambda\in (0,\lambda_{0})$, problem (\ref{proba1}) has at least two nontrivial solutions. 
\label{maintheo33}
\end{theorem}

Problems involving the fractional $p$-Laplacian have been an object of intensive research in the last years in many branches of science such as in phase transition phenomena, population dynamics, and game theory (see \cite{Caf1}, \cite{Ser1}, \cite{Nezza1}, \cite{IannihS}, \cite{global},\cite{hopf},\cite{LindLind1}, \cite{AssoCuesta1}, \cite{fine}, \cite{Lopera1},   \cite{PerYang}) and \cite{IanniPerSqua1}.  In \cite{Vald1}, Valdinoci presents a self-contained exposition on how a simple random walk with possibly long jumps is related to the fractional $p-$Laplacian operator.  For more insights on the applications, we refer to Iannizzotto {\it et al.}\cite{IanniPerSqua1} and Caffarelli \cite{Caf1} where the authors provide 
 a detailed review of current applications and challenges faced when dealing with these nonlocal operators.  

  This paper was motivated by the results obtained by Castro {\it et al.} in \cite{CastroFigLop} for the case of the $p$-Laplacian operator and by Lopera et. al in \cite{Lopera1} for the fractional $p$-Laplacian. In those articles, the authors proved the existence of a positive solution for problem (\ref{proba1}) when the potential 
  $V\equiv 0$. The existence result was obtained by showing that the associated energy functional for problem (\ref{proba1}) had the geometry of the mountain-pass theorem of Ambrosetti-Rabinowitz \cite{AmbRab}. They also proved that the solution was positive by using some new regularity results and Hopf's Lemma.  
 
 The main goal of this work is to extend the results of Lopera {\it et al.} in \cite{Lopera1} by proving the existence of at least two solutions for problem (\ref{proba1}). We will use a variant of the mountain-pass theorem due to Hofer \cite{Hofer2} and infinite-dimensional Morse theory to obtain the existence of a second solution for both cases where $f(0)\ne 0$ and $f(0)=0,$ respectively.

This paper is organized as follows: In Section \ref{prel} we present some preliminary results that will be used throughout this work. In Section \ref{firstex}, we prove that the associated energy functional to problem (\ref{proba1}) has a critical point $u_\lambda$ of mountain-pass type. In Section \ref{secex}, we apply infinite-dimensional Morse theory to compute the critical groups of the associated energy functional at infinity. In Section \ref{SecCriticalGroups}, we compute the critical groups of the associated energy functional for problem (\ref{proba1} at the origin. Finally, we prove the existence and multiplicity results in Section \ref{secproofmainres}.

\section{Preliminaries}\label{prel}

In this work, we will use a variational approach to study the existence and multiplicity of solutions for problem (\ref{proba1}). We start with some notation and preliminary results that will be used throughout this article. 

Let $\Omega$ be an open bounded subset of in $\R^{N}$, $N\geqslant 2$, with boundary $\partial\Omega$.   Denote by $C(\overline{\Omega})$ the set of continuous functions on $\overline\Omega$. The space of $\gamma$-H\"{o}lder continuous functions is defined by
$$C^{\gamma}(\overline{\Omega})=\{u\in C(\overline{\Omega}):[u]_{C^\gamma(\overline{\Omega})}<\infty\},$$
where $0<\gamma \leqslant 1$ and 
$$[u]_{C^\gamma(\overline{\Omega})}=\sup_{x,y\in\overline{\Omega},x\ne y}\frac{|u(x)-u(y)|}{|x-y|^\gamma}.$$
The space $C^\gamma(\overline{\Omega})$ is a Banach space endowed with the norm 
$$\|u\|_{C^\gamma(\overline{\Omega})} = \|u\|_{L^\infty(\Omega)}+[u]_{C^\gamma(\overline{\Omega})}.$$
 

In some of the regularity results that will be used in this article, it will be required that the domain $\Omega\subset\R^N$, $N\geqslant 2$, be a Lipschitz domain. This is the content of the next definition. 

\begin{definition}
{\rm  
 We will say that $\Omega\subset\R^{N}$ has a Lipschitz boundary, and call it a Lipschitz domain, if, for every $x_0\in\partial\Omega$, there exists $r>0$ and a map $h:B_{r}(x_0)\rightarrow B_{1}(0)$ such that
 \begin{enumerate}
     \item [(i)] $h$ is a bijection;
     \item [(ii)] $h$ and $h^{-1}$ are both Lipschitz continuous functions;
     \item [(iii)] $h(\partial\Omega\cap B_{r}(x_0))=Q_0$;
     \item [(iv)] $h(\Omega\cap B_{r}(x_0))=Q_{+}$,
 \end{enumerate}
where $B_{r}(x_0)$ denotes the $n-$dimensional open ball of radius $r$ and center at $x_0\in\partial\Omega$, and 
$$Q_{0}:=\{(x_1,\ldots,x_n)\in B_{1}(0)\mbox{ }| x_n=0\}\mbox{ and } Q_+:=\{(x_1,\ldots,x_n)\in B_1(0)\mbox{ }|x_n > 0\}.$$
}
\end{definition} 

Next, we introduce the space of functions where the associated energy functionals will be defined. Let  $s\in(0,1)$ and $1\leqslant  p < \infty$, and denote by 
\begin{equation}
    W_{0}^{s,p}(\Omega)=\{u\in W^{s,p}(\R^{N}):u=0\mbox{ a.e in }\R^{N}\backslash\Omega\}
    \label{sob0}
\end{equation}
the subset of the following fractional Sobolev space $W^{s,p}(\R^{N})$:
\begin{equation}
    W^{s,p}(\R^{N})=\left\{u\in L^{p}(\R^N):\int_{\R^{2N}}\frac{|u(x)-u(y)|^{p}}{|x-y|^{N+sp}}\,dx\,dy < \infty\right\},
    \nonumber 
\end{equation}
endowed with the norm 
\begin{equation}
    \|u\|_{s,p} := \left(\|u\|_{p}^{p}+[u]_{s,p}^{p}\right)^{1/p},
    \label{defnorm}
\end{equation}
where $\|\cdot \|_{p}$ denotes the norm in $L^{p}(\Omega)$ for $1\leqslant p < \infty$ and 
\begin{equation}
    [u]_{s,p} := \int_{\R^{2N}}\frac{|u(x)-u(y)|^{p}}{|x-y|^{N+sp}}\,dx\,dy,
    \label{gagnorm}
\end{equation}
is the Gagliardo seminorm. It can be shown that $W^{s,p}(\R^{N})$ endowed with the norm $\|\cdot\|_{s,p}$ is a Banach space  and $W_{0}^{s,p}(\Omega)\subset W^{s,p}(\R^N)$ is a closed subspace. In the case $1<p<\infty$, $W^{s,p}(\Omega)$ is a reflexive Banach space (see Asso {\it et al.} \cite[Section 2.1]{AssoCuesta1}).

By virtue of the Sobolev-type inequality (see \cite[Theorem $6.7$]{Nezza1}), it can be shown that the space $W_{0}^{s,p}(\Omega)$ can be endowed with the norm 
\begin{equation}
    \|u\|:=[u]_{s,p},  
    \label{norm1}
\end{equation}
for $s\in(0,1)$ and $1\leqslant p < \infty$. 

We will also denote by $\widetilde{W}^{s,p}(\Omega)$ the following Sobolev space 
    \begin{equation}
        \left\{u\in L_{loc}^{p}(\R^{N}):\exists U\supset\supset\Omega \mbox{ s.t } \|u\|_{W^{s,p}(U)}+\int_{\R^{N}}\frac{|u(x)|^{p-1}}{(1+|x|)^{N+ps}}\,dx < \infty\right\},
        \nonumber 
    \label{wtildedef}
    \end{equation}
where $\Omega\subset\R^{N}$ is a bounded set (see \cite[Definition $2.1$]{global} for more details). Since $\Omega$ is a bounded set, it follows from \cite[Remark $1.1$]{hopf} that $W_0^{s,p}(\Omega)\subset \widetilde{W}^{s,p}(\Omega)$. We will refer to the space $\widetilde{W}^{s,p}(\Omega)$ during the proof of a comparison principle for problem (\ref{proba1}). 

For more details on fractional Sobolev spaces, see  \cite[Section $2$]{Nezza1}, \cite{Ser2}, and references therein.

In this paper, we shall denote by $X$ the fractional Sobolev space $W_{0}^{s,p}(\Omega)$. 

Define $J_\lambda:X\rightarrow\R$, the energy functional associated with problem (\ref{proba1}), by
\begin{equation}
     J_{\lambda}(u) = \frac{1}{p}\|u\|^{p} + \frac{1}{p}\int_{\Omega}V(x)|u|^p\,dx -\lambda\int_{\Omega}F(u)\,dx,\quad\mbox{ for }u\in X,
     \label{funcdef1}
\end{equation}
and $\lambda > 0$ with $\|\cdot\|$ defined in (\ref{norm1}).  

The functional $J_\lambda$ is well-defined and $J_\lambda\in C^1(X,\R)$. It can be shown that the Fr\'{e}chet derivative of $J_\lambda$ is given by 
 \begin{equation}\label{frechet1}
 \begin{aligned} 
    \langle J_{\lambda}^{\prime}(u),\varphi\rangle & = \int_{\R^{2N}}\frac{\Phi_p(u(x)-u(y))(\varphi(x)-\varphi(y))}{|x-y|^{N+sp}}\,dx\,dy \\ & \qquad\quad +  \int_\Omega V(x)|u|^{p-2}u\varphi\,dx - \lambda\int_{\Omega}f(u)\varphi\,dx,
    \end{aligned}
\end{equation}
for all $\varphi\in X$, where $\Phi_p:\R\rightarrow\R$ is given by $\Phi_p(s)=|s|^{p-2}s,$ for $s\in\R$. 

We will say that $u$ is a weak solution of problem (\ref{proba1}) if $u$ is a critical point of $J_\lambda;$ namely, 
\begin{equation}
    \langle J_\lambda^{\prime}(u) ,\varphi\rangle =0,\quad\mbox{ for all }\varphi\in X.
    \label{cpdef}
\end{equation}

For every $1<p_1<p_s^*$, we shall denote by $C_{p_1}$ the optimal constant in the Sobolev embedding theorem; namely, 
\begin{equation} 
\|u\|_{p_1} \leqslant C_{p_1} \|u\|,\quad\mbox{ for all }u\in X,
\label{condd1}
\end{equation}
see \cite[Theorem $6.7$]{Nezza1}.

In the proof of the existence of a solution of a mountain-pass type, we will need the following result due to Lindgren and Lindqvist \cite{LindLind1}. 
\begin{theorem}\cite[Theorem 5]{LindLind1}
There exists a non-negative minimizer $u\in W_{0}^{s,p}(\Omega)$, $u\not\equiv 0$, and $u=0$ in $\R^{N}\backslash\Omega$ of the fractional Rayleigh quotient:
\begin{equation}\label{lambda1Dfn}
    \lambda_1=\inf_{u \in W_{0}^{s,p}(\Omega)\backslash
    \{0\}
}\displaystyle\frac{\displaystyle\int_{\R^{N}}\int_{\R^{N}}\frac{|u(u)-u(x)|^p}{|y-x|^{\alpha p}}\,dx\,dy}{\displaystyle\int_{\R^N}|u(x)|^{p}\,dx}.   
\end{equation}

It satisfies the Euler-Lagrange equation 
\begin{equation}
    \int_{\R^{2N}}\frac{|u(y)-u(x)|^{p-2}(u(y)-u(x))(\varphi(y)-\varphi(x))}{|y-x|^{\alpha p}}\,dx\,dy = \lambda\int_{\R^{N}}|u|^{p-2}u\varphi\,dx,
    \label{EL1}
\end{equation}
with $\lambda=\lambda_1$ whenever $\varphi\in C_{c}^{\infty}(\Omega)$. If $\alpha p> 2N$, the minimizer is in $C^{0,\beta}(\R^{N})$ with $\beta=\alpha -2N/p.$
\label{lindlind1}
\end{theorem}
Theorem \ref{lindlind1} motivated the following definition:

\begin{definition}\cite[Definition $6$]{LindLind1}
    {\rm
   We say that $u\not\equiv 0$, $u\in W_{0}^{s,p}(\Omega)$, $s=\alpha -n/p$, is an {\it eigenfunction}  of $\Omega$, if the Euler-Lagrange equation (\ref{EL1}) holds for all test functions $\varphi\in C_{c}^{\infty}(\Omega)$. The corresponding $\lambda$ is called an {\it eigenvalue. }
    }
\end{definition}

\begin{remark}
    {\rm 
    The minimizer found in Theorem \ref{lindlind1} is called the {\it first eigenfunction} of $((-\Delta)_p^s,W_{0}^{s,p}(\Omega))$.}
\end{remark}

To use some of the minimax theorems in the literature, we have to check that the associated energy functional also satisfies some kind of compactness condition. 

\begin{definition} 
{\rm
 We will say that $(u_{n})\subset X$ is a PS--sequence for $J$ if
$$|J(u_{n})|\leqslant C\quad\mbox{ for all } n, \ \mbox{ and }\  J^{\prime}(u_{n})\rightarrow 0, \mbox{ as }n\rightarrow\infty,$$
where $C$ is a positive constant. We say that a functional $J\in C^1(X,\mathbb{R})$ satisfies the Palais--Smale condition (PS--condition) if any PS-sequence $(u_n)\subset X$ possesses a convergent subsequence.
}
\end{definition}

To prove the existence of a second solution for problem (\ref{proba1}) in Theorem \ref{maintheo}, we will need the concept of critical groups from infinite-dimensional Morse Theory.

 Define $J_{\lambda}^{c}=\{w\in X|\, J_{\lambda}(w)\leqslant c\}$,  the sub--level set of $J_{\lambda}$ at $c$, and set 
$$
    \mathcal{K}=\{u\in X|\, J_{\lambda}^{\prime}(u)=0\} ,
$$ 
the critical set of $J_\lambda$.  For  an isolated critical point  $u_0$  of $J_{\lambda}$, the $q$--critical groups of $J_{\lambda}$ at $u_0$, with coefficients in a field $\F$ of characteristic $0$, are defined by 
\begin{equation}
    C_{k}(J_{\lambda},u_0)=H_{k}(J_{\lambda}^{c_{0}}\cap U,J_{\lambda}^{c_{0}}\cap U\backslash\{u_0\}),\quad\mbox{ for all }k\in\Z,
    \nonumber 
\end{equation}
where $c_{0}=J_{\lambda}(u_0)$, $U$ is a neighborhood of $u_0$ that contains no critical points of $J_{\lambda}$ other than $u_0$, and $H_{*}$ denotes the singular homology groups. The critical groups are independent of the choice of $U$ by the excision property of homology (see Hatcher \cite{AH}). For more information on the definition of critical groups, we refer the reader to \cite{KC}, \cite{PerSch}, \cite{MMP}, and \cite{MW1}. 

Next, we present the concept of the critical groups at infinity introduced by Bartsch and Li in \cite{BLi}. Assume that $J_\lambda\in C^1(X,\R)$ satisfies the Palais-Smale condition. Let $\mathcal{K}=\{u\in X:J_\lambda^{\prime}(u)=0\}$ be the set of critical points of $J_\lambda$ and assume that under these assumptions the critical value set is bounded from below; that is, 
$$a_o < \inf J_\lambda(\mathcal{K}),$$
for some $a_o\in\R.$

The critical groups at infinity are defined by
\begin{equation} 
C_{k}(J_\lambda,\infty) = H_{k}(X,J_\lambda^{a_{o}}),\quad\mbox{ for all }k\in\Z, 
\label{cinfdef}
\end{equation}
(see \cite{BLi}).  These critical groups are well-defined as a consequence of the Second Deformation Theorem (see Perera and Schechter \cite[Lemma $1.3.7$]{PerSch}).  

In this work, we will use the concept of a critical point of a functional being of a mountain-pass type. We use the definition 
 found in Hofer \cite{Hofer2} and Montreanu {\it et al.} in \cite{MMP}.

\begin{definition}\cite[Definition $6.98$]{MMP} Let $X$ be a Banach space, $J\in C^{1}(X,\R)$, and $u_0\in\mathcal{K}.$ We say that $u_0$ is of {\it mountain-pass type} if, for any open neighborhood $U$ of $u_0$, the set $\{w\in U|\, J(w)<J(u_0)\}$ is nonempty and not path-connected.    
\label{mtpassdef}
\end{definition}

The critical groups of mountain-pass type can be described by the following proposition found in Montreanu \textit{et al.}  \cite{MMP}:

\begin{proposition}
    \cite[Proposition $6.100$]{MMP} Let $X$ be a reflexive Banach space, $J\in C^{1}(X,\R),$ and $u_0\in\mathcal{K}$ be isolated with $c=J(u_{o})$ in $J(\mathcal{K}).$ If $u_0$ is of mountain-pass type, then $C_{1}(J,u_0)\not\cong 0.$
    \label{propMMP}
\end{proposition}

Put $\mathcal{K}_d=\{u\in X|J(u)=d,J^{\prime}(u)=0\}$, the critical set at level $d$.
One of the critical points that will be obtained in the proof of Theorem \ref{maintheo} satisfies a variant of the mountain-pass theorem due to Hofer, which we present next for the reader's convenience. 

\begin{theorem}\cite{Hofer2}
Assume that $X$ is a real Banach space. Let $J\in C^{1}(X,\R)$ satisfy the Palais-Smale condition and assume that $e_0$ and $e_1$  are distinct points in $X$. Define 
\begin{equation}\label{def_A}
    A=\{a\in C([0,1],X)|\quad a(i)=e_i,\mbox{ for }i=0,1\}, 
\end{equation}
\begin{equation}
    d=\inf_{a\in A}\sup J(|a|),\quad |a|=a([0,1]),\quad c=\max\{J(e_0),J(e_1)\}.
    \label{defimagea}
\end{equation}
If $d>c$, the set $\mathcal{K}_d$ is non-empty. Moreover, there exists at least one critical point $u_0$ in $\mathcal{K}_d$ that is either a local minimum or of mountain-pass type. If all the critical points in $\mathcal{K}_d$ are isolated in $X$ the set $\mathcal{K}_d$ contains a critical point of mountain-pass type.    
\label{hofmaintheo}
\end{theorem}

\begin{remark}
{\rm Once we prove that the functional $J_\lambda$ defined in (\ref{funcdef1}) satisfies the conditions of Theorem \ref{hofmaintheo} in Section \ref{secex}, for the case $f(0)\ne 0$, assuming that $J_\lambda$ has only one critical point $u_\lambda$, it will follow from Proposition \ref{propMMP} that 
\begin{equation}
    C_{1}(J_{\lambda},u_{\lambda})\not\cong 0.
    \label{cmountaing1}
\end{equation}
 
We do not have information about the other critical groups $C_k(J_\lambda,u_\lambda)$ when $k\ne 1$. But, the fact that $C_{1}(J_\lambda,u_\lambda)$ is nontrivial will be enough to prove the existence of a second critical point for the functional $J_\lambda$. A similar argument will be used in Section \ref{secproofmainres} for the case $f(0)=0$. 

}
\end{remark}
 
Finally, the last result we will need to prove multiplicity results for problem (\ref{proba1}) for the case $f(0)\ne 0$ is found in Bartsch and Li \cite{BLi}. 

\begin{proposition}\cite[Proposition $3.6$]{BLi}
    Suppose that $J\in C^{1}(X,\R)$ satisfies the Palais-Smale condition at level $c$ for every $c\in\R$. If $\mathcal{K}=\emptyset$, then $C_{k}(J ,\infty)\cong 0$ for all $k\in\Z.$ If $\mathcal{K}=\{u_\lambda\}$, then $C_{k}(J ,\infty)\cong C_{k}(J ,u_\lambda)$, for all $k\in\Z.$
    \label{bliprop}
\end{proposition}

We shall prove in Section \ref{secex} that $C_{k}(J_\lambda,\infty)\cong 0$ for all $k>0$; that is, the critical groups of $J_\lambda$ at infinity are all trivial for $k\ne 0.$ In particular, we will have $C_{1}(J_\lambda,\infty)\cong 0$. Hence, assuming, by a way of contradiction, that $J_\lambda$ has only the critical point $u_\lambda$ found in Section \ref{firstex}, we will then obtain a contradiction based on the result of Proposition \ref{bliprop} and the assertion in (\ref{cmountaing1}).  

In the next section, we will prove the existence of a mountain-pass type solution for problem (\ref{proba1}).

\section{Existence and a priori estimates}\label{firstex} 

\subsection{Existence of a mountain-pass type solution}\label{subsec1}

In this section, we will show that the functional $J_\lambda$ defined in (\ref{funcdef1}) satisfies the conditions of the variant of the mountain-pass theorem due to Hofer \cite{Hofer2} as presented in Theorem \ref{hofmaintheo}. 

First, by virtue of conditions (\ref{e11}) and (\ref{e12}), it can be shown that there exists $B_1>0$ such that 
\begin{equation}
    F(s)\leqslant B_{1}(|s|^{q+1}+1),\quad\mbox{ for all }s\in\R. 
    \label{condB1}
\end{equation}

It also follows from (\ref{e11}) and (\ref{e12}) that, for all $s\geqslant 0$, there exists $A_1,C_1 >0$ such that 
\begin{equation}
    F(s) \geqslant A_{1}(s^{q+1}-C_1),\quad\mbox{ for all }s\geqslant 0.
    \label{condA1}
\end{equation}

In what follows, let $r>0$ be the positive number given by 
\begin{equation}
    r = \frac{1}{q+1-p},
    \label{defr}
\end{equation}
where $p,q$ satisfy the conditions in hypothesis $(H_1)$. 

In the next two lemmas, we prove the geometric conditions in Theorem \ref{hofmaintheo}.

\begin{lemma}\label{lema2}
There exist $\tau>0$, $c_1>0$ and  $\hat{\lambda}_2 \in (0,1)$ such that if $\|u\|=\tau \lambda ^{-r}$ then $J_\lambda (u) \geqslant c_1(\tau \lambda^{-r})^p$ for all $\lambda \in (0, \hat{\lambda}_2)$, where $r$ is given in (\ref{defr}).
\end{lemma}
\begin{proof}[Proof:]
By virtue of the Sobolev embedding theorem and hypothesis $(H_3)$, it follows from the definition of $J_\lambda $ in (\ref{funcdef1}) that 
\begin{equation}
\begin{aligned} 
      J_{\lambda}(u) &= \frac{1}{p}\|u\|^p+\frac{1}{p}\int_\Omega V(x) |u|^p dx- \lambda \int_{\Omega}F(u) d x\\
      & \geqslant \frac{1}{p}\|u\|^p-\frac{c_V}{\lambda_1 p} \|u\|^p  - \lambda B_1 C_{q+1}^{q+1} \|u\|^{q+1}- \lambda B_1|\Omega|,    
    \end{aligned}
    \label{eq111}
\end{equation}
for all $u\in X$. 
Let $\tau > 0 $ be a small enough constant such that the following identity is satisfied:
\begin{equation}
    1-\frac{c_{V}}{\lambda_1}=\frac{3}{2}pC_{q+1}^{q+1}B_1\tau ^{q+1-p}.
    \label{cv1cond}
\end{equation} 
Next, setting $\|u\|=\lambda^{-r}\tau$ in (\ref{eq111}) and using the fact that  $r(q+1)+1=-rp$, we get that 
\begin{equation}
    \begin{aligned}
         J_{\lambda}(u) \geqslant  \lambda^{-rp} \left[ \frac{\tau^p}{p}\left(1- \frac{c_V}{\lambda_1}\right) -B_1  C_{q+1}^{q+1} \tau^{q+1} - \lambda^{1+rp} B_1 |\Omega|\right],
    \end{aligned}
    \label{eq112}
\end{equation}
for all $u\in X$.

Then, by virtue of (\ref{cv1cond}), it follows from (\ref{eq112}) that 
\begin{equation} 
         J_{\lambda}(u)   \geqslant  \lambda^{-rp}\left(\frac{1}{2}pC_{q+1}^{q+1}B_1\tau ^{q+1-p}-\lambda^{1+rp}|\Omega|B_1 \right), 
            \label{eq113}
\end{equation}
for all $u\in X$.  

Finally, choose $\lambda \in (0,\hat{\lambda}_2)$ with $\hat{\lambda}_2:=\tau^{p/(1+rp)}(4pB_1|\Omega|)^{-1/(1+rp)}$. Then, for this choice of $\lambda$, we obtain from (\ref{eq113}) that
\begin{equation} 
         J_{\lambda}(u)   \geqslant c_1(\tau\lambda^{-r})^{p};\quad\mbox{ for }u\in X,
            \nonumber 
\end{equation}
where $c_1=\frac{1}{4p}$.   This concludes the proof of the lemma. 
 
\end{proof}

\begin{lemma}\label{lema1}
Let $\varphi_o \in X$ be such that $\varphi_o > 0$ and $\|\varphi_o\|=1$. There exists $\hat{\lambda}_1 >0$ such that if $\lambda \in (0, \hat{\lambda}_1)$ then $J_\lambda (c\lambda ^{-r}\varphi_o) \leqslant 0 $, where $r$ is given in (\ref{defr}).
\end{lemma}

\begin{proof}[Proof:]
Set $\ell=c \lambda^{-r}$, where $c,\lambda >0$ are positive constants to be chosen shortly. Then, by virtue of hypothesis $(H_1)$, the estimate (\ref{condA1}), and the characterization of the first eigenvalue of the fractional $p$-Laplacian from Theorem \ref{lindlind1},  we obtain 
\begin{eqnarray}\label{E<0}
\begin{aligned}
  J_{\lambda}(\ell \varphi_o) & =\frac{1}{p}\|\ell \varphi_o\|^p +\frac{1}{p}\int_\Omega V(x)|\ell\varphi_o|^p dx - \lambda \int_{\Omega} F(\ell \varphi_o)\,dx\\
   & \leqslant \frac{\ell^p}{p}\| \varphi_o\|^p+\frac{\ell^p}{p}\|V\|_\infty\|\varphi_o\|_p^p- \lambda A_1 \ell^{q+1} \int_{\Omega} \varphi_o^{q+1} d x+ \lambda A_1C_1 |\Omega|\\
   & \leqslant \frac{\ell^p}{p}\left(1+\frac{1}{\lambda_1}\|V\|_\infty \|\varphi_o\|^p- \lambda p A_1 \ell^{q+1-p}\|\varphi_o\|_{q+1}^{q+1}\right)+\lambda A_1C_1 |\Omega|.
 \end{aligned}
\end{eqnarray}

Next, define $c>0$ such that 
\begin{equation}
    c^{q+1}=\frac{2c^p}{pA_1\|\varphi_o\|_{q+1}^{q+1}}\left(1+\frac{1}{\lambda_1}\|V\|_\infty \right).
    \label{defcc}
\end{equation}

Then, by virtue of (\ref{defcc}) and the definition of $\ell$, it follows from (\ref{E<0}) that
 \begin{eqnarray}\label{Eqdev1}
\begin{aligned}
  J_{\lambda}(\ell \varphi_o) & \leqslant \lambda^{-rp}\frac{c^p}{p}\left[-\left(1+\frac{1}{\lambda_1}\|V\|_\infty\right)+\lambda^{1+rp}A_1 C_2|\Omega| \right]. \end{aligned}
\end{eqnarray}

Set $\hat{\lambda}_1$ to be 
$$\hat{\lambda}_1 = \left[\frac{1+\frac{1}{\lambda_1}\|V\|_\infty}{2pA_1C_2|\Omega|}\right]^{\frac{1}{1+rp}}.$$
Then, it follows from (\ref{Eqdev1}) that
 \begin{equation*}
     J_{\lambda}(\ell \varphi_o) \leqslant -\frac{c^p}{2p}\lambda^{-rp} \leqslant 0, 
 \nonumber 
 \end{equation*}
 for all $\lambda\in (0,\hat{\lambda}_1)$,
 which establishes the lemma.
\end{proof}

In the next lemma, we will show that the functional $J_\lambda$ satisfies the Palais-Smale condition.

\begin{lemma}\label{PS}
Assume that $(H_1)-(H_3)$ are satisfied and  $\lambda \in (0, \, \lambda_3$) with $\lambda_3:=\min\{\hat{\lambda}_1,\,\hat{ \lambda}_2\}$, where $\hat{\lambda}_1$ is given 
by Lemma \ref{lema1} and  $\hat{\lambda}_2$  is given by Lemma \ref{lema2}.
Then, $J_\lambda$ satisfies the Palais-Smale condition.
\end{lemma}
\begin{proof}[Proof:]

Let $(u_n)$ be a Palais-Smale sequence for $J_\lambda$ in $X$; that is, 
\begin{equation}
    |J_\lambda (u_n)|\leqslant C,\quad\mbox{ for all } n;
    \label{pss1}
\end{equation}
where $C>0$ is a constant and there exists a sequence of positive numbers $(\varepsilon_n)$ such that 
\begin{equation}
    \langle J'_\lambda(u_n),\varphi\rangle \leqslant \varepsilon_n\|\varphi\|, \quad\mbox{ for all }n,
    \label{pss2}
\end{equation}
and all $\varphi\in X$ and $\varepsilon_n\rightarrow 0$ as $n\rightarrow\infty.$

In particular, setting $\varphi=u_n$ in (\ref{pss2}), we get that there exists $N_1 >0$ such that,  
$$|\langle J_\lambda '(u_n), u_n \rangle| \leqslant \|u_n\|,\quad\mbox{ for all }n\geqslant N_1.$$
Hence, we can write 
\begin{equation}\label{ARpss3}
  -\| u_n\|^p-\| u_n\| 
 \leqslant   -\| u_n\|^p+ \langle J'(u_n), u_n \rangle ,
\quad
\mbox{ for } n\geqslant N_1.
\end{equation}
Thus, using the definition of the Fr\'{e}chet 
derivative of $J_\lambda$ given in 
(\ref{frechet1}), and the definition of 
the norm $\|\cdot\|$ given in (\ref{norm1}) and (\ref{gagnorm}), we obtain from (\ref{ARpss3}) that 
\begin{equation}\label{pss3}
-\| u_n\|^p-\| u_n\| \leqslant  
\int_\Omega V(x) |u_n|^p d x - \lambda \int_\Omega f(u_n)u_n \,dx,
\end{equation}
for $n\geqslant N_1.$

On the other hand, using the estimate in (\ref{ambcond}) in  hypothesis $(H_2)$, we have that 
$$
\frac{1}p\| u_n\|^p-\frac{\lambda}\theta  \int_\Omega f(u_n)u_n \,dx +\frac{\lambda}{\theta} K|\Omega| \leqslant      \frac{1}p\| u_n\|^p-\lambda  \int_\Omega F(u_n) \,dx 
$$
for $n\in \mathbb{N}$; so that, using the 
definition of the $J_\lambda$ in 
(\ref{funcdef1}), 
\begin{equation}\label{ARpss4}
\frac{1}p\| u_n\|^p-\frac{\lambda}\theta  \int_\Omega f(u_n)u_n \,dx +\frac{\lambda}{\theta} K|\Omega| \leqslant      
J_\lambda(u_n) -\frac{1}{p}\int_\Omega V(x) |u_n|^p \ dx,
\end{equation}
for all $n\in\mathbb{N}$.

Now, it follows from (\ref{ARpss4}) and 
the hypothesis in   (\ref{pss1}) that 
\begin{equation}\label{pss4}
\frac{1}p\| u_n\|^p-\frac{\lambda}\theta  \int_\Omega f(u_n)u_n \,dx +\frac{\lambda}{\theta} K|\Omega|
 \leqslant C-\frac{1}p\int_\Omega V(x)|u_n|^p \,dx, 
\end{equation}
for $n\in \mathbb{N}$.

Next, multiply on both sides of the estimate in
(\ref{pss3}) by $\dfrac{1}{\theta}$ and add
$
\displaystyle
\frac{1}{p} \|u_n\|^p 
$
on both sides of the ineqaulity to obtain
\begin{equation}
    \begin{aligned}
   \left( \frac{1}p -\frac{1}{\theta}  \right)\| u_n\|^p - \frac{1}\theta\|
   u_n\|
   &\leqslant  \frac{1}p \| u_n\|^p+\frac{1}{\theta}\left( \int_\Omega V(x) |u_n|^p dx - \lambda \int_\Omega f(u_n)u_n  dx\right),
\end{aligned}
\label{pss5}
\end{equation}
for $n\geqslant N_1$.

Now, it follows from the estimate (\ref{pss4}) 
that 
\begin{equation}\label{ARpss5}
\frac{1}p \| u_n\|^p
	+\frac{1}{\theta}
\left( 
\int_\Omega V(x) |u_n|^p dx 
	- \lambda \int_\Omega f(u_n)u_n  dx
\right)
\leqslant
C - \left( \frac{1}{p} - \frac{1}{\theta}\right)
\int_\Omega V(x) |u_n|^p dx,
\end{equation}
for $n\in\mathbb{N}$.

Consequently, combining the estimates in 
(\ref{pss5}) and (\ref{ARpss5}), 
\begin{equation}\label{ARpss25}
   \left( \frac{1}p -\frac{1}{\theta}\right)\| u_n\|^p 
   	- \frac{1}\theta\|u_n\|
   	\leqslant 
   		C 
   		- \left( \frac{1}{p} - \frac{1}{\theta}\right)
   			\int_\Omega V(x) |u_n|^p dx,
\end{equation}
for $n\geqslant N_1$.

Next, use the estimate for the potential $V$ in
hypothesis $(H_3)$ to obtain from (\ref{ARpss25}) that 
\begin{equation}\label{pss6}
\left( \frac{1}p -\frac{1}{\theta}  \right)\| u_n\|^p 
	- \frac{1}\theta\| u_n\|
   \leqslant  
   C+\left( \frac{1}p -\frac{1}{\theta}\right)
   \frac{c_{V}}{\lambda_1}\|u_n\|^p,
    \quad\mbox{ for } n\geqslant N_1,
\end{equation}
where we have also used the definition of $\lambda_1$ in (\ref{lambda1Dfn}). 

Rearranging (\ref{pss6}) we obtain 
\begin{equation}
\left( \frac{1}p -\frac{1}{\theta}  \right)\left( 1-\frac{c_{V}}{\lambda_1}  \right)\|u_n\|^p- \frac{1}{\theta} \|u_n\| \leqslant C,
    \quad\mbox{ for } n\geqslant N_1,
    \nonumber 
\end{equation}
from which we obtain that $(u_n)$ is bounded in $W_0^{s,p}(\Omega)$. 

Hence, since $(u_n)$ is bounded in $X$, we may invoke
the Banach-Alaoglu theorem (see \cite[Theorem $2.18$]{Lieb}) to deduce, passing to a 
subsequence if necessary, that there exists $u\in X$ such that  
\begin{equation} 
\begin{aligned}
	u_{n} & \rightharpoonup u\quad\mbox{ weakly in }X\quad 
    \mbox{ as } n\to\infty.
              \end{aligned} 
              \nonumber 
\end{equation}
Furthermore, since $1<q+1<p^*$, by virtue of the Sobolev embedding theorem, we can also assume that
\begin{equation} 
\begin{aligned}
	 u_n & \rightarrow u\quad\mbox{ in }L^{q+1}(\Omega)\quad 
    \mbox{ as } n\to\infty \\
     u_{n}(x)&\rightarrow u(x)\quad\mbox{ a.e. in }\Omega\quad\mbox{ as }n\rightarrow\infty.
              \end{aligned} 
              \label{pss13}
\end{equation}

Next, put $q_1^\prime=\frac{q+1}{q};$ so that, $q^\prime>1$, and $q^\prime q=q+1.$ Hence, by virtue of estimate (\ref{condB1}) we get
\begin{equation}
    \begin{aligned}
    |f(u_n)|^{q_1'} &\leqslant  B_1^{q_1'} (|u_n|^q+1)^{q_1'} \\
    &\leqslant C_1 (|u_n|^{qq_1'}+1) \\
    &\leqslant  C_1 (|u_n|^{q+1}+1),
\end{aligned}
\label{estt1}
\end{equation}
for all $n\in\N$, where $C_1$ is a positive constant.

Thus, applying H\"{o}lder's inequality with exponent $q_1^{\prime}$ in (\ref{estt1}) and its conjugate, we obtain 
\begin{equation} 
\begin{aligned}
\lambda \int_\Omega f(u_n)(u_n-u) \,dx & \leqslant C(\|u_n\|_{q+1}+1)\|u_n-u\|_{q+1}\\ &
\leqslant C\|u_n-u\|_{q+1} ,
\end{aligned}
\nonumber 
\end{equation}
where $C$ is a positive constant.

Consequently, letting $n\rightarrow\infty$ in the previous estimate and applying (\ref{pss13}) with the Lebesgue dominated convergence theorem,  we get 
\begin{equation}\label{f(u_n)(u_n-u)}
    \lambda \int_\Omega f(u_n)(u_n-u) d x\to 0, \quad \text{as $n\to \infty$ }.
    \end{equation}
    
Next, put $p^\prime=\dfrac{p}{p-1}$ (recall that we are assuming $p>1$); so that $p^\prime>1$  and $p^\prime(p-1)=p.$ Then, by virtue of H\"{o}lder's inequality we have
\begin{equation} 
\begin{aligned}
\int_\Omega |V(x) ||u_n|^{p-1}|u_n-u|\,dx &\leqslant \|V \|_\infty \|u_n\|^{p-1}_p\|u_n-u\|_p \\
&\leqslant C \|u_n-u\|_p\\
&\leqslant C \|u_n-u\|_{q+1},
\end{aligned}
\nonumber 
\end{equation}
for all $n\in\N$, where $C$ is a positive constant.

Hence, letting $n\rightarrow\infty$ in the previous estimate and applying (\ref{pss13}) with the Lebesgue dominated convergence theorem, we obtain
\begin{equation}\label{nabla(u_n)(u_n-u)}
\int_\Omega |V(x) ||u_n|^{p-1}|u_n-u|\,dx \to 0, \quad \text{as $n\to \infty$ }.\end{equation}

Next, since $(u_n)$ is a Palais-Smale sequence in $X$, it follows from (\ref{pss2}), (\ref{f(u_n)(u_n-u)}), and (\ref{nabla(u_n)(u_n-u)} that 
\begin{equation}\label{eq13}
\lim_{n\to \infty} \int_{\mathbb{R}^{2N}} \frac{\Phi_p(u_n(x)-u_n(y))((u_n-u)(x)-(u_n-u)(y))}{|x-y|^{N+sp}}\,dx=0.
 \end{equation}

Once again, using the fact that $u$ is the weak limit of $u_n$ we have
\begin{equation}\label{eq14}
    \lim_{n \to \infty}\int_{\mathbb{R}^{2N}} \frac{\Phi_p(u(x)-u(y))((u_n-u)(x)-(u_n-u)(y))}{|x-y|^{N+sp}} \,dx=0.
\end{equation}

On the other hand, it follows from an application of the H\"{o}lder's inequality as in \cite[Lemma $3$]{Lopera1} that
\begin{equation}
\begin{aligned}
   & \int_{\Omega} \frac{\Phi_p(u_n(x)-u_n(y))-\Phi_p(u(x)-u(y))}{|x-y|^{N+sp}}   ((u_n-u)(x)-(u_n-u)(y))\,dx\,dy  & \\
&    = \int_{\Omega} \left[\frac{|u_n(x)-u_n(y)|^{p}}{|x-y|^{N+sp}}- \frac{\Phi_p(u_n(x)-u_n(y))(u(x)-u(y))}{|x-y|^{N+sp}} \right.\\
 &   \left.-\frac{\Phi_p(u(x)-u(y))(u_n(x)-u_n(y))}{|x-y|^{N+sp}}+ \frac{|u(x)-u(y)|^p}{|x-y|^{N+sp}} \right] \,dx\,dy\\
 & \geqslant \|u_n\|^p-\|u_n\|^{p-1}\|u\|-\|u_n\|\|u\|^{p-1}+\|u\|^p\\
 & = (\|u_n\|^{p-1}-\|u\|^{p-1})(\|u_n\|-\|u\|). 
\end{aligned}
\label{pss09}
\end{equation} 

Then, in view of the fact that 
$$\left(\|u_n\|^{p-1}-\|u\|^{p-1}\right)(\|u_n\|-\|u\|\geqslant 0, \quad\mbox{ for all }n,$$
it follows from \eqref{eq13},\eqref{eq14}, and (\ref{pss09}) that
\begin{equation}
\lim_{n\to \infty}    ([\|u_n\|^{p-1}-\|u\|^{p-1})(\|u_n\|-\|u\|)=0,
\nonumber 
\end{equation}
from which we get 
\begin{equation}
    \lim_{n\to \infty} \|u_n\|=\| u\|.
    \label{pss11}
\end{equation}

Finally, by virtue of (\ref{pss11}) and the fact that $u_n \rightharpoonup u$ weakly in $X$, we conclude that $u_n \to u$ strongly in $X$. Hence, $J_\lambda$ satisfies the Palais-Smale condition.
 
\end{proof}

Next, we present the main result of this section.

\begin{theorem}\label{existance-sol-mpt}
    Assume that the hypotheses $(H_1)$-$(H_3)$ are satisfied. Then, for $\lambda$ sufficiently small,  the functional $J_\lambda$ has a  critical point $u_\lambda \in X$ of mountain-pass type.
    Moreover, \begin{equation}
        c_1\lambda^{-rp} \leqslant J_\lambda (u_\lambda) \leqslant c_2 \lambda ^{-rp},
        \label{eslp1}
    \end{equation} 
    where $c_1$ and $c_2$ are positive constants independent of $\lambda$, and $r$ is given in (\ref{defr}).
    \label{maintheo1}
\end{theorem}
\begin{proof}[Proof:]
    It follows from Lemma \ref{E<0}, Lemma \ref{lema1}, and Lemma \ref{PS}, that, for any $\lambda \in (0, \lambda_3)$, the functional $J_\lambda$ defined in (\ref{funcdef1}) satisfies the conditions of Theorem \ref{hofmaintheo}. Therefore, $J_\lambda$ possesses a critical point, $u_\lambda$, with critical value characterized by $$J_\lambda (u_\lambda)=\inf_{a\in A}\max J_\lambda (|a|),$$
    with
    $$A=\{a\in C([0,1],X)|a(0)=0,\quad a(1)=c\lambda^{-r}\varphi_{o}\},$$
    where $a(1)$ is obtained in Lemma \ref{lema1} and $|a|=a([0,1]).$

    Furthermore, by virtue of Lemma \ref{lema1}, we observe that 
$$J_\lambda(sc\lambda^{-r}\varphi_o) \leqslant c_2\lambda^{-rp},\quad\mbox{ for all }0\leqslant s\leqslant 1,$$ 
where $c_0$ is a positive constant independent of $\lambda$. Hence, we conclude that
$$J_\lambda (u_\lambda) \leqslant c_2\lambda ^{-rp}.$$

Finally, it follows from Lemma \ref{lema2} that there exists a positive constant $c_1$ independent of $\lambda$ such that $$c_1\lambda ^{-rp} \leqslant J_\lambda (u),\quad\mbox{ for all }\|u\|=\tau\lambda^{-r}.$$ 
Then, it follows from the characterization of the critical value that $$c_1\lambda ^{-rp} \leqslant J_\lambda (u_\lambda).$$
This concludes the proof of the theorem. 
\end{proof}

The next two results will be used in the proof of a comparison principle for problem (\ref{proba1}).   

 \begin{lemma}\label{norm-estimate}
    Assume that the hypotheses $(H_1)$-$(H_3)$ are satisfied and let $u_\lambda$ be the mountain-pass critical point of $J_\lambda$
    given in Theorem \ref{maintheo1}.  There exists a constant $c$ such that 
     \begin{equation}
         \|u_\lambda \| \leqslant c\lambda ^{-r}.
         \label{reg1}
     \end{equation}
and $r$ is given in (\ref{defr}).
 \end{lemma}
\begin{proof}[Proof:]
Let $u_\lambda$ be a critical point of $J_\lambda$ given by Theorem \ref{maintheo1}. Then, it follows from (\ref{cpdef}) that 
\begin{equation}
    \langle J_\lambda^{\prime}(u) ,\varphi\rangle =0,\quad\mbox{ for all }\varphi\in X.
    \label{cpdef1}
\end{equation}

Then, setting $\varphi=u_\lambda$ in (\ref{cpdef1}) and using (\ref{frechet1}), we get
\begin{equation}
    \|u_\lambda \|^p+\int_{\mathbb{R}^N} V(x)|u_\lambda|^p dx= \lambda \int_{\Omega} f(u_\lambda)u_\lambda  dx.
    \nonumber 
\end{equation}

It then follows from the Ambrosetti-Rabinowitz type condition in \eqref{ambcond} that 
\begin{eqnarray}
    \begin{aligned}
    \left( \frac{1}p-\frac{1}\theta \right) \|u_\lambda\|^p & =    \frac{1}p\|u_\lambda \|^p-\frac{1}\theta \left(\lambda \int_{\Omega} f(u_\lambda)u_\lambda  dx-\int_{\mathbb{R}^N} V(x)|u_\lambda|^p dx \right)\\
    & \leqslant    \frac{1}p\|u_\lambda \|^p-\frac{\lambda}\theta \left(\int_{\Omega}  \theta F(u_\lambda)  dx+K|\Omega| \right) +\frac{1}\theta \int_{\mathbb{R}^N} V(x)|u_\lambda|^p dx
    \\
    &\leqslant \frac{1}p\|u_\lambda \|^p-\lambda \int_{\Omega}   F(u_\lambda)  dx +\frac{1}p \int_{\mathbb{R}^N} V(x)|u_\lambda|^p dx
    -\frac{\lambda K}\theta |\Omega| \\
    & \leqslant J_\lambda (u_\lambda)+C\lambda^{-rp}; 
    \end{aligned}
    \nonumber 
\end{eqnarray}
so that, using (\ref{eslp1}) in Theorem \ref{existance-sol-mpt}, (\ref{reg1}) follows.
\end{proof}

Finally, we present a lower and upper estimates for 
 $\|u_\lambda\|_{\infty}$, where $u_\lambda$ is the critical point obtained in Theorem \ref{maintheo1}.  These results  will be used on the proof of comparison principle for problem (\ref{proba1}).

\begin{lemma}\label{linfty-estimate}
Assume that the hypotheses $(H_1)$-$(H_3)$ are satisfied. 
Let $u_\lambda$ be a weak solution of problem
(\ref{proba1}) obtained via Theorem \ref{existance-sol-mpt} and $\lambda_3$ be as in Lemma \ref{PS}. Then, there exists a constant $C$  such that, for all $0<\lambda<\lambda_3$, 
\begin{equation}
    C\lambda^{-r} \leqslant \|u_\lambda\|_\infty ,
    \label{uinftest}
\end{equation}    
and $r$ is given in (\ref{defr}).
\end{lemma}
\begin{proof}[Proof:]
By virtue of the estimate in (\ref{eslp1}) for $J_\lambda (u_\lambda)$ in Theorem \ref{existance-sol-mpt}, and the fact that $\min F>-\infty $,  we get 
    \begin{eqnarray}\label{l-infty1}
        \begin{aligned}
            \lambda \int_\Omega f(u_\lambda)u_\lambda dx & = \|u_\lambda\|^p+\int_{\R^N} V(x)|u|^p dx \\
            &= pJ_\lambda (u_\lambda) dx + p\lambda \int_\Omega F(u_\lambda) dx \\
            &\geqslant pC\lambda ^{-rp}+p\lambda |\Omega| \min F \\
            &\geqslant C\lambda ^{-rp}.
        \end{aligned}
    \end{eqnarray}
    On the other hand, by virtue of the growth of $f$ in (\ref{e11}), we get 
    \begin{equation}\label{l-infty2}
            \lambda \int_\Omega f(u_\lambda)u_\lambda dx \leqslant B\lambda \|u_\lambda\|_\infty ^{q+1}.
    \end{equation}
   Combinining the estimates \eqref{l-infty1} and \eqref{l-infty2},  we obtain (\ref{uinftest}). 
\end{proof}

In the proof of the comparison principle, we will need the following regularity result found in Mosconi et.al\cite{PerYang}. 

 \begin{lemma}\cite[Lemma $2.3$]{PerYang}\label{lemma_esti_Linfty}
     Let $g\in L^t(\Omega)$, $N/(sp)<t\leqslant \infty$ and $u\in W_0^{1,p}(\Omega)$ be a weak solution of $(-\Delta)_s^p=g$ in $\Omega$. Then 
     \begin{equation}
         \|u\|_\infty \leqslant C\|g\|_t^{1/(p-1)}.
         \nonumber 
     \end{equation}
 \end{lemma}

The following theorem due to Ianizotto et. al \cite{fine} establishes a sharp boundary regularity result for the fractional $p-$Laplacian, for $p\geqslant 2.$  The assumption of $p\geqslant 2$ will allow us to obtain enough regularity up to the boundary of $\Omega$ to get a positive solution for problem (\ref{proba1}).   

\begin{theorem}\cite[Theorem $1.1$]{fine}\label{t4.1}
Let $p \geqslant 2$, $\Omega$ be a bounded domain with $C^{1,1}$ boundary and $\text{d}(x) = \text{dist}(x, \partial \Omega)$. There exist $\alpha \in (0,s)$ and $C>0$ depending on $N, \Omega, p$ and $s$, such that, for all $g \in L^{\infty}(\Omega)$, a weak solution $u \in W_0^{s,p}(\Omega)$ of problem 

\begin{equation}
\left\{\begin{aligned} 
    (-\Delta)_p^s(u) &=  g; \quad \mbox{ in }\Omega, \\
     u & = 0\quad\mbox{ in }\R^N \backslash \Omega ,
\end{aligned}\right.   
\nonumber 
\end{equation}
satisfies $u/\text{d}^s \in C^{\alpha}(\overline{\Omega})$ and 
\begin{equation}
    \left\|\frac{u}{\text{d}^s} \right\|_{C^{\alpha}(\overline{\Omega})} \leqslant C \|g\|_{\infty}^{\frac{1}{p-1}}.
    \nonumber 
\end{equation}
\end{theorem}

Finally,  we present the last result of this section that will be used to prove the existence of a positive solution for problem (\ref{proba1}). 

\begin{lemma}\label{L_inf_bound-above}
Assume that the hypotheses $(H_1)$-$(H_3)$ are satisfied. 
 Let $\lambda_{3} >0$ be as in Lemma \ref{PS}. Then, there exist $\alpha \in (0, s]$ and a constant $C>0$ such that, for all $0<\lambda <\lambda_3$, the solution  $u_\lambda$ given in Theorem \ref{existance-sol-mpt} of the problem \eqref{maintheo} satisfies $u_\lambda /d^s \in C^\alpha (\overline{\Omega})$. Furthermore 
 \begin{equation}
    \|u_\lambda\|_\infty \leqslant C \lambda^{-r}
    \nonumber 
\end{equation}
 and 
    \[ \bigg\| \frac{u_\lambda}{d ^s}\bigg\|_{C^{\alpha}(\overline{\Omega})} \leqslant C \lambda^{-r} . \]    
    and $r$ is given in (\ref{defr}).
\end{lemma}
\begin{proof}[Proof:]
It follows from the assumption $Nq/(sp)<p_s^{*}$ that there exists $t>1$ such that $\frac{N}{sp}<t$  and $tq<p_s^*$, which implies $t(p-1)<p_s^*$. Set $g:=\lambda f\circ u_\lambda+V\Phi_p(u_\lambda).$ Since $W_0^{s,p}(\Omega) \hookrightarrow L^{tq}(\Omega)$ is a continuous embedding and  $|g|\leqslant A_1 \lambda (|u_\lambda|^{q}+1)+\|V\|_\infty |u_\lambda|^{p-1}$ we obtain
\begin{equation}
\begin{aligned}
    \int_{\Omega}|\lambda f(u_{\lambda})(x)+V\Phi_p(u_\lambda)|^t\,dx & \leqslant \lambda^t \int_{\Omega}|A_1(u_{\lambda}^q+1)|^t \,dx+\|V\|_\infty^t\int_{\Omega}|u_{\lambda}|^{t(p-1)}\,dx\\
    & \leqslant \lambda^t C \int_{\Omega}(|u_{\lambda}|^{qt}+1) \,dx+\|V\|_\infty ^t\int_{\Omega}|u_{\lambda}|^{t(p-1)}\,dx.
\end{aligned}
\nonumber 
\end{equation}
Hence, $g\in L^{t}(\Omega)$ and it follows from Lemma \ref{lemma_esti_Linfty} that 
\begin{equation}\label{eq1-lemma4}
    \|u_\lambda\|_\infty \leqslant \|g\|_t^{\frac{1}{p-1}} .\end{equation} 

On the other hand, by virtue of Lemma \ref{norm-estimate}, we have
\begin{equation}
\begin{aligned}
  \|g\|_t &\leqslant C_1\lambda \|u_\lambda\|_{tq}^q + C_2 \|u_\lambda\|_{t(p-1)}^{p-1}\\ &  \leqslant C_1\lambda \|u_\lambda\|^q+C_2  \|u_\lambda\|^{p-1}\\&  \leqslant    C(\lambda^{1-rq}+\lambda^{-r(p-1)}). 
\end{aligned}
\nonumber 
\end{equation}
Therefore, we obtain from \eqref{eq1-lemma4} and the fact $-r=(1-rq)/(p-1)$ that
\begin{equation}\label{upper-bound-u-infty}
    \|u_\lambda\|_\infty \leq \|g\|_t^{1/(p-1)}\leq C \lambda^{-r}.
\end{equation}

Thus, $u_\lambda \in L^{\infty}(\Omega)$ and then $g \in L^\infty(\Omega)$. Hence, by virtue of Theorem \ref{t4.1}, there exists $\alpha \in (0, s]$ and $C>0$, depending only on $N,p,s$ and $\Omega$, such that the solution $u_\lambda $ satisfies $u_\lambda /d  ^s \in C^\alpha (\overline{\Omega})$ and 
    \[\bigg\| \frac{u_\lambda}{d ^s}\bigg\|_{C^{\alpha}(\overline{\Omega})} \leqslant C\| g \|^\frac{1}{p-1}_{\infty} \leqslant \lambda^{-r}.\]
\end{proof}


\subsection{Existence of a positive solution}
To prove that the solution $u_\lambda$ found in Subsection  \ref{subsec1} is positive, we will list two results found in Del Pezzo et. al \cite{hopf} and one theorem due to Ianizzoto {\it et al.} \cite{global}, which will lead us to a comparison principle for the fractional $p$-Laplacian problem in (\ref{proba1}).   

First, we recall two basic definitions that will be used in this section for the reader's convenience. 

\begin{definition}
    {\rm
   Let $\Omega\subset\R^{N}$, $N\geqslant 1$, be an open set. We say that $x_o\in\partial\Omega$ satisfies the interior ball condition if there is $x\in\Omega$ and $r>0$ such that 
   $$B_r(x)\subset\Omega,\quad\mbox{ and }x_o\in\partial B_{r}(x),$$
where $B_{r}(x)=\{z\in\R^N: |z-x|<r\}.$
    
    }
\end{definition}

Next, we recall the concept of a function $u\in \widetilde{W}^{s,p}(\Omega)$ being a super-solution of the fractional $p$-Laplacian problem  (\ref{proba1}).

\begin{definition}
    Let $\Omega\subset \R^{N}$ be an open bounded set with $N\geqslant 1.$ We say that $u\in \widetilde{W}^{s,p}(\Omega)$ is a super-solution of (\ref{proba1}) if 
     \begin{eqnarray}
    \int_{\R^{2N}}\frac{\Phi_p(u(x)-u(y))(\varphi(x)-\varphi(y))}{|x-y|^{N+sp}} dx dy + \int_\Omega V(x)|u|^{p-2}u\varphi dx \geqslant \lambda\int_{\Omega}f(u)\varphi dx,\nonumber
\end{eqnarray}
for each $\varphi\in \widetilde{W}^{s,p}(\Omega).$
\label{supersoldef}
\end{definition}

The next two theorems due to Del Pezzo et. al \cite{hopf} will play a special role in the main result to be discussed in this section.

\begin{theorem}\cite[Theorem $1.4$]{hopf} \label{t4.2}
    Let $c \in C(\overline{\Omega})$ be a non-positive function and $u \in \widetilde{W}^{s,p}(\Omega) \cap C(\overline{\Omega})$ be a weak super-solution of 
\begin{equation}\label{e4.21}
        (-\Delta)_p^su = c(x)|u|^{p-2}u \quad \text{in}~~ \Omega.
    \end{equation}
 If $\Omega$ is bounded and $u \geqslant 0$ a.e. in $\mathbb{R}^N \backslash \Omega$, then either $u>0$ in $\Omega$ or $u=0$ a.e. in $\mathbb{R}^N$.
\end{theorem}

\begin{theorem}\cite[Theorem $1.5$]{hopf}\label{t4.3}
    Let $\Omega$ satisfy the interior ball condition at $x_0 \in \partial \Omega$, $c \in C(\overline{\Omega})$, and $u \in \widetilde{W}^{s,p}(\Omega) \cap C(\overline{\Omega})$ be a weak super-solution of \eqref{e4.21}. Suppose that $\Omega $ is bounded, $c(x) \leqslant 0$ in $\Omega$ and $u \geqslant 0$ a.e. in $\mathbb{R}^N \backslash\Omega$.  Then, either $u = 0$ a.e. in $\mathbb{R}^N$, or 
    \begin{equation}
        \liminf_{\substack{x
        \rightarrow x_0\\x\in B}} \frac{u(x)}{(d(x))^s}>0,
        \label{reftod}
    \end{equation}
    where $B \subseteq \Omega$ is an open ball in $\Omega$, such that $x_0 \in \partial B$, and $d$ is the distance from $x$ to $\mathbb{R}^N\backslash B$.
\end{theorem}

Next, we present a version of the comparison principle for problem (\ref{proba1}) motivated by a result due to Lindgren et.al \cite[Lemma $9$]{LindLind1} (see also Ianizzotto {\it et al.} in \cite[Proposition $2.10$]{global}).

\begin{theorem}\label{comparison-extended}
    Let $\Omega$ be a bounded set of $\R^N$, $N\geqslant 2$, and $u,v \in \widetilde{W}^{s,p}(\Omega)$ satisfy $u \leqslant  v$ in $\mathbb{R}^N \backslash \Omega$. Moreover, assume that
   \begin{equation} 
    \begin{aligned}\int_{\mathbb{R}^{2N}} &\frac{\Phi_p(u(x)-u(y))(\varphi(x)-\varphi(y))}{|x-y|^{N+sp}}\,dx\,dy +\int_{\R^N}V(x)\Phi_p(u) \varphi dx\\
        & \leqslant  \int_{\mathbb{R}^{2N}} \frac{\Phi_p(v(x)-v(y))(\varphi(x)-\varphi(y))}{|x-y|^{N+sp}}\,dx\,dy +\int_{\R^N}V(x)\Phi_p(v)\varphi dx,
    \end{aligned}
    \label{cpeq1}
    \end{equation}
    for all $\varphi \in W_0^{s,p}(\Omega)$, $\varphi \geqslant 0$ a.e in $\Omega$. If $V(x)\geqslant 0$ for a.e. $x\in\R^{N}$, then $u \leqslant v$ in $\Omega$.
\end{theorem}
\begin{proof}[Proof:]
Set $\varphi=(u-v)^+$, where $(u-v)^{+}=\max\{u-v,0\}$ denotes the positive part of the function $u-v,$ in (\ref{cpeq1}) to get
\begin{eqnarray}\label{compar-1}
\begin{aligned}
&\int_{\R^N}V(x)(\Phi_p(u) -\Phi_p(v))(v-u)^+(x) dx\\
        & \leqslant \int_{\mathbb{R}^{2N}} \frac{(\Phi_p(v(x)-v(y))-\Phi_p(u(x)-u(y)))((v-u)^+(x)-(v-u)^+(y))}{|x-y|^{N+sp}}\,dx\,dy.
\end{aligned}    \end{eqnarray}

Using the identity found in \cite[page $809$]{LindLind1}, 
\begin{equation}
    \Phi_p (b)-\Phi_p (a)=(p-1)(b-a)\int_{0}^1 |a+t(b-a)|^{p-2}\, dt,
    \nonumber 
\end{equation}
with $b=u(x)$ and $a=v(x)$, we obtain the following estimate
\begin{eqnarray}
\begin{aligned}
0&\leqslant(p-1)(u(x)-v(x))(u-v)^+(x)\int_{0}^1 |v(x)+t(u(x)-v(x))|^{p-2}\, dt \\
&=  (\Phi_p (u(x))-\Phi_p (v(x)))(u-v)^+(x).
\end{aligned}
\end{eqnarray}
for a.e. $x\in \R^N$.

Hence, we conclude that the left hand side of \eqref{compar-1} is nonnegative.  The remaining proof of the theorem follows the same line of reasoning as in  \cite[Lemma $9$]{LindLind1} and we omit its proof. 
\end{proof}

Next, we proceed to show that the solution $u_\lambda$ found through Theorem \ref{maintheo1} is positive in $\Omega.$  

\begin{theorem} Assume that $p\geqslant 2$ and $V(x)\geqslant 0$ for a.e $x\in\Omega.$ If $p-1<q<\min\{\frac{sp}{N}p_s^*, \, p_s^*-1\}$, then there exists $\lambda^* >0$ such that, for all $0<\lambda < \lambda^*$, problem \eqref{proba1} has at least one positive solution $u_\lambda \in C^{\alpha}_o(\overline{\Omega})$ for some $0<\alpha < 1$.
\label{theo56}
\end{theorem}

\begin{proof}[Proof:]
 From Lemma \ref{existance-sol-mpt} we know that, for any $\lambda \in (0, \lambda_3)$, there exists a solution $u_\lambda \in X$.
Assume, by a way of contradiction, that there exists a sequence $(\lambda_j)_{j\in\N}$ with $0<\lambda_j<1$ such that $\lambda_j\rightarrow 0$ as $j\rightarrow\infty$ and, for all $j\in\N$, we have 
\begin{equation}
    |\Omega_j| > 0,
\end{equation}
where $\Omega_j=\{x\in\Omega|u_{\lambda_j}(x)\leqslant 0 \}$, for all $j\in\N$,  and $|\Omega_j|$ denotes the Lebesgue measure of the set $\Omega_j$.

Set $w_j=\frac{u_{\lambda_j}}{\|u_{\lambda_j}\|_\infty}.$ Notice that $w_j(x)\leqslant 0$ for all $x\in\Omega_j.$
Thus, by virtue of the regularity result in \cite[Theorem $1.1$]{global}, we get 
\begin{equation}
    (-\Delta)_p^s (w_j)=h_j(x, w_j),
    \nonumber 
\end{equation}
where 
$h_j(x,s):=-V(x)\Phi_p(s)+\lambda_j \|u_{\lambda_j}\|_\infty ^{1-p}f(\|u_{\lambda_j}\|_\infty s)$.

Using the fact $\lambda_j \|u_{\lambda_j}\|_\infty ^{1-p}<1$ and, by virtue of Lemma \ref{L_inf_bound-above}, $\lambda_j \|u_{\lambda_j}\|_\infty ^{q+1-p}<C$ for $j$ large, and $1-r(1-p+q)=0$, we get 
\begin{eqnarray}
\begin{aligned}
|h_j(x,s)|&\leqslant  |V(x)||s| ^{p-1}+\lambda_j \|u_{\lambda_j}\|_\infty ^{1-p}B((\|u_{\lambda_j}\|_\infty |s|)^{q}+1)\\
&\leqslant  \|V\|_\infty |s| ^{p-1}+B\lambda_j \|u_{\lambda_j}\|_\infty ^{1-p+q} |s|^{q}+B \lambda_j \|u_{\lambda_j}\|_\infty ^{1-p} \\
&\leqslant  \|V\|_\infty |s| ^{p-1}+B\lambda_j ^{1-r(1-p+q)} |s|^{q}+B \\
&\leqslant C_1 |s|^{p^*-1}+C_2.
\end{aligned}
\nonumber 
\end{eqnarray}

By virtue of Theorem \ref{t4.1}, there exists $\alpha \in (0,s]$ such that
\begin{eqnarray}
\begin{aligned}
    \bigg\|\frac{w_j}{d^s_\Omega }\bigg\|_{C^\alpha(\overline{\Omega})} & \leqslant \|h_j(x, w_j)\|^{1/(p-1)}_\infty \\
&\leqslant (C_1 \|w_j\|_\infty^{p^*-1}+C_2)^{1/(p-1)}\\
&= C_3,
 \end{aligned}
 \label{az1} 
\end{eqnarray}
where $C_3$ is a positive constant which does not depend on $\lambda_j$. 

Next, choose $\beta$ such that $0<\beta <\alpha$. By virtue of Arzel\`{a}-Ascoli Theorem (see \cite[Theorem 40 on pg. 169]{Royden}), up to a subsequence, it follows from (\ref{az1}) that 
$$\lim_{j\to \infty} \frac{w_j}{d^s_\Omega } = \frac{w}{d^s_\Omega },\quad \mbox{ in }   C^{\beta} (\overline{\Omega}).$$ 

The next step consists of using the comparison principle to prove that $w(x)\geqslant 0$. In fact, let $v_0\in W_0^{s,p}(\Omega)$ be a solution of 
	\begin{equation}  \left\{\begin{aligned} 
		(-\Delta)_p^su+V(x)\Phi_p(u) &=1, \quad \mbox{ in }\Omega;\\
		u&=0,   \quad\mbox{in $\R^N\backslash\Omega$},
	\end{aligned}\right.
 \nonumber 
 \end{equation}
obtained in Appendix \ref{appendix}.

Let $K_j=\frac{\lambda_j}{\|u_{\lambda_j}\|^{p-1}_{\infty}}\displaystyle\min_{t\in \mathbb{R}} f(t).$  Observe that $K_j<0$.
Next, let $v_j=-(-K_j)^{1/(p-1)}v_0$. Then $v_j$ solves
		\begin{equation} \left\{\begin{aligned}
			(-\Delta)_p^s u +V(x)\Phi_p(u)&=K_j, \quad\mbox{ in } \Omega;\\
			u&=0,  \quad\mbox{ in }\R^N\backslash\Omega.
		\end{aligned}\right.
  \nonumber 
  \end{equation}

Next, observe that, for all $\varphi \in W_0^{s,p}(\Omega)$ with $\varphi \geqslant 0$, we have the following estimate
\begin{equation}
\begin{split} &\int_{\mathbb{R}^{2N}}\frac{\Phi_p(w_j(x)-w_j(y))}{|x-y|^{N+sp}}(\varphi(x)-\varphi(y)) dx dy+\int_{\Omega}V(x)\Phi_p(w_j) \varphi dx\\
    &= \int_{\Omega}\lambda_jf(u_{\lambda_j})\|u_{\lambda_j}\|_{\infty}^{1-p} \varphi dx \\
    & \geqslant \int_{\Omega} K_j \varphi dx\\
    &= \int_{\mathbb{R}^{2N}} \frac{\Phi_p(v_j(x)-v_j(y))}{|x-y|^{N+sp}}(\varphi(x)-\varphi(y)) dx dy+\int_{\Omega}V(x)\Phi_p(v_j) \varphi dx.
    \end{split}   
    \label{eqff1}
\end{equation}

The estimate (\ref{eqff1}) implies that $(-\Delta)_p^s(w_j) \geqslant (-\Delta)_p^s(v_j)$. By virtue of the comparison principle stated in Theorem \ref{comparison-extended},  we conclude that  $w_j\geqslant v_j$. 
Since $v_j\to 0$, as $j\to \infty$, we obtain  $w (x)\geqslant 0$, for $x\in\Omega.$

Next, let $t:=Npr/(N-sp)>1$.  By virtue of Lemma \ref{linfty-estimate}) and Lemma \ref{L_inf_bound-above}, we have that 
$$C_1\lambda^{-r} \leqslant \|u_\lambda \|_\infty \leqslant C_2\lambda^{-r}.$$ 
Then, we get
\begin{equation}
    \begin{aligned}
     \lambda _j |f(u_{\lambda_j}(x))| \|u_{\lambda _j}\|_\infty ^{1-p} & \leqslant C\lambda _j( |u_{\lambda_j}(x)|^q+1) \|u_{\lambda _j}\|_\infty ^{1-p} \\
     & \leqslant C\lambda _j( \|u_{\lambda_j}\|_\infty ^q+1) \|u_{\lambda _j}\|_\infty ^{1-p} \\
     & \leqslant C\lambda_j (\lambda_j^{-rq}+1)\lambda_j^{r(p-1)} \\
     & \leqslant C\lambda_j \lambda_j^{-rq}\lambda_j^{r(p-1)} \\
     &= C\lambda_j^{1-rq+r(p-1)}\\ &=C,
     \end{aligned}
     \nonumber 
     \end{equation}
where $C$ is a positive constant and  $q<p_s^*-1$. 

Hence, it follows from the previous estimate that 
		\[\int_{\Omega} (\lambda _j f(u_j) \|u_{\lambda _j}\|_\infty ^{1-p})^t\, dx  \leqslant C_{ \Omega} | \Omega|.\]
		
  Thus, $\{\lambda _j f(u_j) \|u_{\lambda_j}\|_\infty ^{1-p}\}_{j}$ is bounded in $L^t(\Omega)$ and we may assume that it converges weakly in $L^t(\Omega)$. Let $z:=\lim_{j\rightharpoonup 0 } \lambda _j f(u_j) \|u_{\lambda_j}\|_\infty ^{1-p}$ be its weak limit.
		Since $f$ is bounded from below and  $\lim_{j \to \infty} \lambda _j \|u_{\lambda_j}\|_\infty ^{1-p} =0,$ it follows that $z \geqslant 0.$ 
  
  We claim that $(-\Delta)_p^s(w)=z$. In fact, by virtue of Lemmas \ref{norm-estimate} and \ref{linfty-estimate}, we can follow the same line of reasoning as in the proof of  \cite[Theorem $1.1$] {Lopera1}  to obtain
	\begin{align}\label{w-weak-sol}
			\begin{split}
	&			\lim_{j \to \infty} \int_{\mathbb{R}^{2N} }  \frac{|w_j(x)-w_j(y)|^{p-2}(w_j(x)-w_j(y))(\varphi (x)-\varphi (y))}{|x-y|^{N+sp}}  d x \, d y \\ 
    & = \int_{\mathbb{R}^{2N}  }  \frac{|w(x)-w(y)|^{p-2}(w(x)-w(y))(\varphi (x)-\varphi (y))}{|x-y|^{N+sp}}  d x \, d y .
			\end{split}
		\end{align}
  
  On the other hand, since  $w_j\to w$ uniformly in $\overline{\Omega}$ and $w \in L^p(\Omega)$, we also have that 
  \begin{equation}\label{w-weak-sol2}
      \lim_{j\to \infty} \int_\Omega V(x)\Phi_p(w_j)\varphi (x) dx = \int_\Omega V(x)\Phi_p(w)\varphi (x) dx. 
  \end{equation}
  Notice that $w_j\to w$ in $W_0^{s,p} (\Omega),$ which implies that  $w\in W_0^{s,p} (\Omega)$. Consequently, by virtue of \eqref{w-weak-sol}, \eqref{w-weak-sol2}, and the fact that $z$ is the weak limit of $\{\lambda_j f(u_{\lambda_j}\|u_{\lambda_j}\|_\infty^{1-p}\}$, we have that 
    $$(-\Delta)_p^s w+V\Phi_p(w)=z.$$
   That is, $w$ is a weak supersolution of $(-\Delta)_p^s w+V\Phi_p(w)=0$. Hence, by virtue of Theorem \ref{t4.2} and Theorem \ref{t4.3}, we have two alternatives:  First, $w=0$ cannot hold since $w_j\to w$ in $C^{\beta}(\overline{\Omega})$ and $\|w_j\|_\infty=1$ for all $j\in\N$. Second,  $w>0$ in $\Omega$ and,  for all $x_0 \in \partial \Omega$, 
		\begin{equation} 
		\liminf_{\substack{x
        \rightarrow x_0\\x\in B}} 
        \frac{w(x)}{(d(x))^s}>0 ,
  \nonumber\end{equation}
  where $B \subseteq \Omega$ is an open ball in $\Omega$, such that $x_0 \in \partial B$, and $d(x)$ is the distance from $x$ to $\mathbb{R}^N\backslash B$ 
  (see (\ref{reftod})). 
		
Therefore, there exist $j_o$ sufficiently large such that , for all $j\geqslant j_o$, we have $w_j>0$.  But this contradicts the fact that $w_j(x)=\frac{u_{\lambda_j}(x)}{\|u_{\lambda_j\|_{\infty}}} \leqslant 0$, for $x\in\Omega_{j}$. 

Hence, $|\Omega_j|=0$ for all $j\in\N$ and we conclude that problem (\ref{proba1}) has at least one positive solution $u_\lambda \in C^{\alpha}_0(\overline{\Omega})$, for some $0<\alpha < 1$.
\end{proof}

\section{Computation of the Critical Groups at Infinity}\label{secex}

In this section, we will obtain the first multiplicity result for problem (\ref{proba1}). The first step will consist of computing the critical groups of $J_\lambda$ at infinity as defined in (\ref{cinfdef}). This will require to use the concept of two topological spaces being homotopically equivalent.  

To show that two topological spaces $A$ and $B$ are homotopically equivalent, denoted by $A \cong B$,  one needs to show that there exist functions $\eta:A\rightarrow B$ and $i:B\rightarrow A$ such that $\eta\circ i\approx id_B$ and $i\circ \eta \approx id_A$, where $id$ denotes the identity function and the symbol $\approx$ denotes the existence of a homotopy.  

In particular, if $B\subset A$ and $i:B\rightarrow A$ denotes the inclusion function and $\eta: A\rightarrow B$ is a deformation retraction from $A$ onto $B$, then we have that $\eta \circ i \approx id_{B}$ and $i\circ\eta = id_{A}$. Hence, to obtain the critical groups of $J_\lambda$ at infinity,  we will prove the existence of a deformation retract from $J_\lambda^{-M}$ onto $S^{\infty}$, for some $M$ to be chosen soon, where $S^{\infty}$ denotes the unit sphere in $X$. Finally, the result will follow by using an argument with the long exact sequence of the topological pair $(X,J^{-M})$ and the fact that $S^{\infty}$ is contractible in $X$.

Let $S^{\infty}=\{u\in X:\|u\|=1\}$ be the unit sphere in $X$. Notice that, for $u\in S^{\infty}$, we have that 
\begin{equation}
    \lim_{t\rightarrow\infty}J_\lambda(tu)=-\infty.
    \label{jinf1}
\end{equation}
In fact, substituting (\ref{condA1}) into (\ref{funcdef1}) and applying hypothesis $(H_3)$, we obtain
\begin{equation}
    J_{\lambda}(tu) \leqslant \frac{t^{p}}{p}\left(1+\|V\|_{L^{\infty}}\|u\|_{p}^{p}\right)  - \lambda A_{1}t^{q+1}\|u\|_{q+1}^{q+1} + \lambda A_{1}C_{1}|\Omega|,
    \label{e14}
\end{equation}
for all $u\in S^{\infty}$. Then, since $p<q+1$, the result (\ref{jinf1}) follows by letting $t\rightarrow\infty$ in (\ref{e14}). 
 
\begin{lemma}
    Assume that the hypotheses $(H_1)-(H_2)$ are satisfied. Then, there exists $\widetilde{M}>0$ such that, for all $M  \geqslant \widetilde{M}$, $J_\lambda^{-M}$ is homotopically equivalent to $S^{\infty}$. 
\end{lemma}
\begin{proof}[Proof:]

We will follow a line of reasoning similar to that presented by Wang in \cite[Section $3$]{Wang} to show  the existence of a deformation retract from $J_\lambda^{-M}$ onto $S^{\infty}$.

First, notice that the critical value set $J_\lambda(\mathcal{K})$ is bounded from below. In fact, if $u_0\in\mathcal{K},$ then setting  $\varphi=u_0$ in (\ref{frechet1}), we get 
\begin{equation}
    \|u_0\|^{p} + \int_{\Omega}V(x)|u_0|^{p}\,dx =\lambda\int_{\Omega}f(u_0)u_0\,dx.
    \label{cond1}
\end{equation}
Next, substitute (\ref{cond1}) into  (\ref{funcdef1}) and use hypothesis $(H_2)$ to obtain 
 \begin{equation}
\begin{aligned} 
    J_{\lambda}(u_0) & = \frac{\lambda}{p}\int_{\Omega}\left[f(u_0)u_0-pF(u_0)\right]\,dx\\
    &\geqslant \frac{\lambda}{p}\int_{\Omega}\left[(\theta-p)F(u_0)+K\right]\,dx\\
    &\geqslant \frac{\lambda}{p}|\Omega|((\theta-p)\min F +K)=:-a_0,
    \end{aligned} 
    \nonumber 
\end{equation} 
for all $u_0\in\mathcal{K},$ and therefore
\begin{equation}
    -a_o \leqslant \inf J_{\lambda}(\mathcal{K}).
    \label{aodef}
\end{equation}

By virtue of (\ref{jinf1}), given $u\in S^{\infty}$ and $M_1>0$, there exists $t_0=t_0(u)\geqslant 1$ such that,  
\begin{equation}
    J_\lambda (tu) < -M_1,\quad\mbox{ for }t_{0}\geqslant 1, u\in S^{\infty}.
    \nonumber 
\end{equation}
 Define $\widetilde{M}=\min\{-a_o,-M_1\}$. Then, choose $M_2>\widetilde{M}$ such that, for $tu\in J_\lambda^{-M_2}$, we have
\begin{equation}
    J_\lambda(tu) = \frac{t^p}{p}\left(1+\int_{\Omega}V(x)|u|^{p}\,dx\right) - \lambda\int_{\Omega}F(tu)\,dx,
    \label{eq45}
\end{equation}
for $t\geqslant 1$.

Using the chain rule, and taking into account the fact that $f(s)s$ is bounded from below, and $\dfrac{p}{\theta}<1$, it follows from (\ref{eq45}) and hypothesis $(H_2)$ that 
 \begin{equation}
    \begin{aligned}
        \frac{d}{dt}J_\lambda(tu)  
                & = \frac{1}{t}\left[pJ(tu) + \lambda\int_{\Omega}(pF(tu)-f(tu)tu)\,dx\right]\\
&\leqslant \frac{1}{t}\left[-pM_2 + \lambda\int_{\Omega}\left(\frac{p}{\theta}f(tu)(tu)-f(tu)tu-\frac{Kp}{\theta}\right)\,dx\right]
\\
&\leqslant \frac{1}{t}\left[-pM_2 + \lambda\left(\frac{p}{\theta}-1\right)\int_{\Omega}f(tu)(tu))\,dx-\lambda\frac{Kp}{\theta}|\Omega|\right]
\\
&\leqslant \frac{1}{t}\left[-M_{o}+\hat{K}\lambda\right], 
    \end{aligned}
        \label{df111}
\end{equation}
where $M_o$ and $\hat{K}$ are positive constants, for all $tu\in  J_\lambda^{-M_2}.$ 

Choosing $\lambda$ small enough in (\ref{df111}), we get
\begin{equation}
    \frac{d}{dt}J_\lambda(tu) < 0,
    \label{df1}
\end{equation}
for $tu\in J_\lambda^{-M_2},$ and $t\geqslant 1.$

 Let us take $M\geqslant\widetilde{M}$. Then, combining (\ref{jinf1}) and (\ref{df1}), we can invoke the intermediate value theorem to conclude that there exists $T(u)\geqslant 1$ such that 
\begin{equation}
    J_\lambda(T(u)u)=-M,\quad\mbox{ for }u\in S^{\infty}.
    \nonumber 
\end{equation}
  
It also follows from the implicit function theorem \cite[Theorem $15.1$]{Deim1} that $T\in C(S^{\infty},\R)$. 

Finally, let $B^{\infty}=\{u\in X:\|u\|\leqslant 1\}$ be the unit ball in $X$. Define $\eta:[0,1]\times (X\backslash B^{\infty})\rightarrow X\backslash B^{\infty}$ by
\begin{eqnarray}
    \eta(t,u) = (1-t)u + tT(u)u,
    \nonumber 
\end{eqnarray}
for $t\in[0,1]$ and $u\in X\backslash B^{\infty}$. Observe that $\eta(0,u)=u$ and $\eta(1,u)\in J_\lambda^{-M}.$ Then, $\eta$ is a deformation retract from $X\backslash B^{\infty}$ onto $J_{\lambda}^{-M}$. Since $X\backslash B^{\infty}\cong S^{\infty}$, then we conclude that 
$$J_\lambda^{-M}\cong X\backslash B^{\infty}\cong S^{\infty};$$
that is, $J_{\lambda}^{-M}$ is homotopically equivalent to $S^{\infty}$.

\end{proof}

Since $J_\lambda^{-M}$ and $S^{\infty}$ are homotopically equivalent, as shown in the previous lemma, we conclude that the homology groups $H_{k}(J_\lambda^{-M})$ and $H_{k}(S^{\infty})$ are isomorphic, for all $k\in\Z$ (see \cite[Corollary $2.11$]{AH}). Since $S^{\infty}$ is also contractible in $X$ (see Benyamini-Sternfeld  \cite{BS1}), we obtain  that the singular homology groups $H_{k}(J_\lambda^{-M})$ have the homology type of a point for all $k\in\Z$; namely,
\begin{equation}
    H_{k}(J_\lambda^{-M})\cong\delta_{k,0}\F,\quad\mbox{ for all }k\in\Z. 
    \nonumber 
\end{equation}
Using an argument similar to that in \cite[Section $3$]{RecRumb5} with the long exact sequence of reduced homology groups of the topological pair $(X,J_\lambda^{-M})$ and the fact that $J_\lambda$ satisfies the Palais-Smale condition shown in Lemma \ref{PS}, we conclude that the critical groups of $J_\lambda$ at infinity are given by
\begin{equation}
    C_{k}(J_\lambda,\infty)=H_{k}(X,J_\lambda^{-M})\cong\delta_{k,0}\F,\quad\mbox{ for all }k\in\Z.
    \label{cinftyg}
\end{equation}

\section{Computation of the Critical Groups at the Origin}\label{SecCriticalGroups}

In this section, we study the questions of existence and multiplicity for the case $f(0)=0$. In this case, the function $u\equiv 0$ is also a critical point of $J_\lambda$ and we need to obtain some information about the critical groups of $J_\lambda$ at the origin. To obtain another solution, we will need to make an additional assumption about the behavior of $F$ at the origin. This is the content of the next lemma.  

\begin{lemma}
    Assume that the nonlinearity $f$ satisfies $(H_1)$ and its primitive $F$ satisfies the condition
    \begin{equation}
        \limsup_{s\rightarrow 0} \frac{F(s)}{|s|^{p}}=0. 
        \label{limsupF}
    \end{equation}
    Then, the origin is a local minimizer of the functional $J_\lambda$ and its critical groups are given by
    \begin{equation}
        C_{k}(J_\lambda,0)\cong\delta_{k,0}\F,\quad\mbox{ for all }k\in\Z.
        \label{coriging}
    \end{equation}
\end{lemma}

\begin{proof}[Proof:]
By condition (\ref{limsupF}), given any $\varepsilon > 0$, there exists a $\delta>0$ such that 
\begin{equation}
    |s|<\delta \Rightarrow F(s) < \varepsilon|s|^p.
    \label{ineq1}
\end{equation}
It follows from (\ref{condB1}) that there exists a constant $K_1=K_1(\delta)$ such that 
\begin{equation}
    |F(s)|\leqslant K_{1}|s|^{q+1},\quad\mbox{ for all }|s|\geqslant\delta.
    \label{ineq2}
\end{equation}
In fact, assume $s\geqslant\delta$ and use hypothesis $(H_1)$ to get 
\begin{equation}
    |F(s)|\leqslant\int_{0}^{s}|f(\xi)|\,d\xi \leqslant Bs + \frac{B}{q+1}s^{q+1};
    \nonumber 
\end{equation}
so that 
\begin{equation}
    |F(s)|\leqslant B\left[\delta\left(\frac{s}{\delta}\right)+\frac{\delta^{q+1}}{q+1}\left(\frac{s}{\delta}\right)^{q+1}\right].
    \label{ineq4}
\end{equation}
Since we are assuming that $s\geqslant\delta$; so that $\frac{s}{\delta}\geqslant 1$; it follows from (\ref{ineq4}) that 
\begin{equation}
    |F(s)|\leqslant B \left[\delta\left(\frac{s}{\delta}\right)^{q+1}+\frac{\delta^{q+1}}{q+1}\left(\frac{s}{\delta}\right)^{q+1}\right],\quad\mbox{ for }s\geqslant\delta,
    \label{ineq5}
\end{equation}
from which we obtain that 
\begin{equation}
    |F(s)| \leqslant \frac{B}{\delta^{q+1}}\left[\delta+\delta^{q+1}\right]s^{q+1},\quad\mbox{ for }s\geqslant\delta,
    \label{ineq6}
\end{equation}
where we have used the fact that $q+1>p>1$, in view of hypothesis $(H_1)$. Setting $K_1=K_1(\delta)=\frac{B}{\delta^{q+1}}\left[\delta+\delta^{q+1}\right]$, we see that (\ref{ineq2}) follows from (\ref{ineq6}). The case for $s\leqslant -\delta$ is analogous. Hence, the estimate (\ref{ineq2}) is valid  for all $|s|\geqslant\delta.$ 

Next, combine the estimates (\ref{ineq1}) and (\ref{ineq2}) to get 
\begin{equation}
    F(s)\leqslant\varepsilon|s|^p + K_1|s|^{q+1},\quad\mbox{ for }s\in\R.
    \label{ineq7}
\end{equation}
Then, it follows from (\ref{ineq7}) that
\begin{equation}
    \int_{\Omega}F(u)\,dx \leqslant \varepsilon\int_{\Omega}|u|^p\,dx + K_1\int_{\Omega}|u|^{q+1}\,dx;
    \nonumber 
\end{equation}
so that, using the Sobolev inequality \cite[Theorem $6.7$]{Nezza1}, it follows from the previous estimate that 
\begin{equation}
    \int_{\Omega}F(u)\,dx \leqslant C_3\left(\varepsilon+K_1\|u\|^{q+1-p}\right)\|u\|^p,
    \label{ineq9}
\end{equation}
for some positive constant $C_3$. 

Setting $\rho=\displaystyle\left(\frac{\varepsilon}{2K_1}\right)^{1/(q+1-p)}$, we obtain from (\ref{ineq9}) that 
\begin{equation}
    \|u\| < \rho \Rightarrow \int_{\Omega}F(u)\,dx \leqslant C_3\varepsilon \|u\|^p.
    \label{ineq10}
\end{equation}

Next, it follows from the definition of $J_\lambda$ in (\ref{funcdef1}), (\ref{ineq10}), and hypothesis $(H_3)$ that 
\begin{equation}
    J_\lambda (u) \geqslant  \left(\frac{1}{p}-C_3\lambda\varepsilon\right)\|u\|^{p} -   \frac{c_V}{p}\|u\|_{p}^{p}.
    \label{ineq11}
\end{equation}

On the other hand, it follows from  \cite[Section $3$]{LindLind1} that the first  eigenvalue $\lambda_1$ of $(-\Delta)_{p}^{s}$ is characterized by the minimization of the 
Rayleigh quotient: 
\begin{equation}
    \lambda_1 = \inf_{u\in X\backslash\{0\}} \frac{\|u\|^{p}}{\|u\|_p^{p}},
    \label{ineq12}
\end{equation}
with $\lambda_1\in (0,\infty)$.  
See Theorem 5 in \cite{LindLind1}.

Hence, applying (\ref{ineq12}) in (\ref{ineq11}), we get
\begin{equation}
       J_\lambda (u) \geqslant  \left[\frac{1}{p}\left(1- \frac{c_V}{\lambda_1}\right)-C_3\lambda\varepsilon\right]\|u\|^{p}.
    \label{ineq13}
\end{equation}
By virtue of hypothesis $(H_3)$, $c_{V}<\lambda_1$, thus we can choose $\varepsilon > 0$ such that 
\begin{equation}
   \varepsilon < \frac{1}{2pC_3\lambda}\left(1-\frac{c_V}{\lambda_1}\right).
    \label{epschoice}
\end{equation}
Then, by virtue of (\ref{epschoice}), we obtain from (\ref{ineq13}) that
\begin{equation}
    J_\lambda(u)\geqslant \frac{1}{2pC_3\lambda}\left(1-\frac{c_V}{\lambda_1}\right)\|u\|^{p} > J(0),\quad\mbox{ for }0<\|u\|<\rho,
    \nonumber 
\end{equation}
where $\rho>0$ is sufficiently small. Consequently, $u=0$ is a local minimum of $J_\lambda$ in $B_{\rho}(0)$. It follows from (\cite[Example $1$,page $33$]{KC} that 
\begin{equation}
    C_{k}(J_\lambda,0)\cong\delta_{k,0}\F,\quad\mbox{ for }k\in\Z.
    \nonumber 
\end{equation}

\end{proof}

\section{Proofs of the Main Results}\label{secproofmainres}

In this final section of the paper, we present proofs of the main results; namely the proofs the of Theorem \ref{maintheo}, Theorem \ref{maintheo2}, and Theorem \ref{maintheo33}. 

\subsection{Proof of Theorem \ref{maintheo}}
Assume, by a way of contradiction, that $\mathcal{K}=\{u_\lambda\}$ where $u_\lambda$ is the mountain-pass type solution found in Theorem \ref{maintheo1}. Then, it follows from Proposition \ref{propMMP} that
\begin{equation}
    C_{1}(J_\lambda,u_\lambda)\not\cong 0.
    \label{cmnt1}
\end{equation}
Since we are assuming that $\mathcal{K}=\{u_\lambda\}$, we can invoke Proposition \ref{bliprop} to get 
\begin{equation}
    C_{k}(J_\lambda,\infty)\cong C_{k}(J_\lambda,u_\lambda),\quad\mbox{ for all }k\in \Z.
    \label{cg1}
\end{equation}
In particular, if $k=1$ in (\ref{cg1}), we obtain from (\ref{cinftyg}) and (\ref{cmnt1}) that 
$$0\cong C_{1}(J_\lambda,\infty) \cong C_1(J_\lambda,u_\lambda)\not\cong 0,$$
which is a contradiction. 

Therefore, $J_\lambda$ must have at least two critical points and this concludes the proof of the theorem.

\subsection{Proof of Theorem \ref{maintheo2}}
  By virtue of   Theorem \ref{maintheo}, we obtain the existence of two solutions for problem (\ref{proba1}). Furthermore, one of them is of mountain pass type.
  
  Next, assume that $p\geqslant 2$ and $V(x)\geqslant 0$ for a.e $x\in\Omega$. For the case of $f(0)>0$, it follows from the Comparison Theorem \ref{theo56} that both solutions are positive. For the case of $f(0)<0$, Theorem \ref{theo56} leads us to the positivity of the mountain-pass type solution. 

\subsection{Proof of Theorem \ref{maintheo33}}
Assume, by a way of contradiction, that $\mathcal{K}=\{0,u_\lambda\}$, where $u_\lambda$ is the mountain-pass type solution found in Theorem \ref{maintheo1}.  Then, it follows from 
 \cite[Theorem $4.2$, page $35$]{KC} that 
\begin{equation}
    H_{k}(X,J_\lambda^{-M}) \cong C_{k}(J_\lambda,0)\oplus C_{k}(J_\lambda,u_\lambda),\quad\mbox{ for all }k\in\Z.
    \label{hkarg1}
\end{equation}
In particular, setting $k=1$ in (\ref{hkarg1}) and using (\ref{cmountaing1}), (\ref{cinftyg}), and (\ref{coriging}), we obtain 
\begin{equation}
       0 \cong C_{1}(J_\lambda,0)\oplus C_{1}(J_\lambda,u_\lambda) \cong 0\oplus  C_{1}(J_\lambda,u_\lambda)\not\cong 0,
    \nonumber
\end{equation}
which is a contradiction. Therefore, the critical set $\mathcal{K}$ must have at least three critical points. This concludes the proof of the theorem.

\section{Appendix}\label{appendix}
In this section, we prove that the problem
\begin{equation}  \left\{\begin{aligned} 
		(-\Delta)_p^su+V(x)\Phi_p(u) &=1, \quad \mbox{ in }\Omega;\\
		u&=0,   \quad\mbox{in $\R^N\backslash\Omega$},
	\end{aligned}\right.
 \label{pprob1}
 \end{equation}
has a positive weak solution. We will show that the associated energy functional to problem (\ref{pprob1}) is coercive and weakly lower semi-continuous. Then, the existence result follows by a result found in Evans \cite[Theorem $2$, Chapter $8$]{EV}. 

In fact, the associated functional to problem (\ref{pprob1}) is given by
\begin{equation}
    E(u):=\frac{1}{p}\|u\|_{s,p}^p+\frac{1}{p}\int_\Omega V (x)|u|^p\,dx - \int_\Omega u\,dx, \qquad u\in X.
    \label{deffu}
\end{equation}
To prove the coercivity of $E$, let  $(u_n)_n$ be a sequence in $X$ such that $\|u_n\|_{s,p}\to \infty$ as $n\to \infty$. 
 From \eqref{condd1} we have that $\|u_n\|_1 \leqslant C_1\|u_n\|_{s,p}$, for all $n$. Moreover, $\|u_n\|^p_p \leqslant \frac{1}{\lambda_1}\|u_n\|^p_{s,p}$, for all $n$. Therefore, applying these estimates and hypothesis $(H_3)$ into (\ref{deffu}) we get 
 \begin{equation}
 \begin{aligned}
     E(u_n) & \geqslant  \frac{1}{p}\|u_n\|_{s,p}^p-\frac{c_V}{p}\|u_n\|_p^p  - C_1\|u_n\|_{s,p}  \\
     &\geqslant  \frac{1}{p}\left( 1-\frac{c_V}{\lambda_1} \right)\|u_n\|_{s,p}^p - C_1\|u_n\|_{s,p},
     \end{aligned}
     \label{finaleq}
 \end{equation}
for all $n\in\N$. 

Since $1-\frac{c_V}{\lambda_1}>0$ and $p>1$, we obtain from (\ref{finaleq}) that $E(u_n)\to \infty$ as $n\to \infty$.
Now, $E$ is continuous due to its differentiability. Moreover, a simple computation shows that the functional $E$ is convex. Therefore, $E$ is weakly lower semicontinuous (see for example \cite[Theorem $1.5.3$]{Badiale}). This proves that problem \eqref{pprob1} has at least one solution $u\in X$, which is nontrivial. 

Finally, notice that $u$ is a weak supersolution of the problem $(-\Delta)_p^su+V(x)\Phi_p(u) = 0,$
  in $\Omega$,	with	$u=0$,  in $\R^N\backslash\Omega$. Thus, by virtue of Theorem \eqref{t4.2}, it follows that $u>0$.

\bibliographystyle{alpha}
\bibliography{RumbosRecova}

\begin{thebibliography}{IMJNS20}

\bibitem[ACDL24]{AssoCuesta1}
O.~Asso, M.~Cuesta, J.~T. Doumat\`e, and L.~Leadi.
\newblock Maximum and anti-maximum principle for the fractional {$p$}-{L}aplacian with indefinite weights.
\newblock {\em J. Math. Anal. Appl.}, 529(1):Paper No. 127626, 19, 2024.

\bibitem[AR73]{AmbRab}
A.~Ambrosetti and P.~Rabinowitz.
\newblock Dual variational methods in critical point theory and applications.
\newblock {\em Journal of Functional Analysis}, 14:349--381, 1973.

\bibitem[BL02]{BLi}
T.~Bartsch and S.~Li.
\newblock Critical point theory for asymptotically quadratic functionals and applications to problems with resonance.
\newblock {\em Nonlinear Analysis TMA}, 28(3):419--441, 2002.

\bibitem[BRS16]{Ser2}
G.~M. Bisci, V.~D. Radulescu, and R.~Servadei.
\newblock {\em Variational methods for nonlocal fractional problems}, volume 162 of {\em Encyclopedia of Mathematics and its Applications}.
\newblock Cambridge University Press, Cambridge, 2016.

\bibitem[BS83]{BS1}
Y.~Benyamini and Y.~Sternfeld.
\newblock Spheres in infinite-dimensional normed spaces are {L}ipschitz contractible.
\newblock {\em Proc. Amer. Math. Soc.}, 88(3):439--445, 1983.

\bibitem[BS10]{Badiale}
M.~Badiale and E.~Serra.
\newblock {\em Semilinear elliptic equations for beginners: existence results via the variational approach.}
\newblock Springer, 2010.

\bibitem[Caf12]{Caf1}
L.~Caffarelli.
\newblock Non-local diffusions, drifts and games.
\newblock In {\em Nonlinear partial differential equations}, volume~7 of {\em Abel Symp.}, pages 37--52. Springer, Heidelberg, 2012.

\bibitem[CdFL16]{CastroFigLop}
A.~Castro, D.~G. de~Figueredo, and E.~Lopera.
\newblock Existence of positive solutions for a semipositone {$p$}-{L}aplacian problem.
\newblock {\em Proc. Roy. Soc. Edinburgh Sect. A}, 146(3):475--482, 2016.

\bibitem[Cha93]{KC}
K.~C. Chang.
\newblock {\em Infinite-dimensional {M}orse {T}heory and multiple solution problems}, volume~6 of {\em Progress in Nonlinear Differential Equations and their Applications}.
\newblock Birkhauser, 1993.

\bibitem[Dei09]{Deim1}
K.~Deimling.
\newblock {\em Nonlinear {F}unctional {A}nalysis}.
\newblock Dover {P}ublications, 2009.

\bibitem[DNPV12]{Nezza1}
E.~Di~Nezza, G.~Palatucci, and E.~Valdinoci.
\newblock Hitchhiker's guide to the fractional {S}obolev spaces.
\newblock {\em Bull. Sci. Math.}, 136(5):521--573, 2012.

\bibitem[DP17]{hopf}
A.~Del~Pezzo, L.M.;~Quaas.
\newblock A {H}opf's lemma and a strong minimum principle for the fractional $p$-{L}aplacian.
\newblock {\em J. Differential Equations}, 263(1):765--778, 2017.

\bibitem[Eva98]{EV}
L.~C. Evans.
\newblock {\em Partial Differential Equations}, volume~19 of {\em Graduate Studies in Mathematics}.
\newblock American Mathematical Society, 1998.

\bibitem[Hat10]{AH}
A.~Hatcher.
\newblock {\em Algebraic {T}opology}.
\newblock Cambridge University Press, 2010.

\bibitem[Hof85]{Hofer2}
H.~Hofer.
\newblock A geometric description of the neighbourhood of a critical point given by the mountain-pass theorem.
\newblock {\em J. London Math. Soc. (2)}, 31(3):566--570, 1985.

\bibitem[ILPS16]{IanniPerSqua1}
A.~Iannizzotto, S.~Liu, K.~Perera, and M.~Squassina.
\newblock Existence results for fractional {$p$}-{L}aplacian problems via {M}orse theory.
\newblock {\em Adv. Calc. Var.}, 9(2):101--125, 2016.

\bibitem[IMJNS20]{fine}
A.~Iannizzotto, S.~Mosconi, and M.~J.~N.~Squassina.
\newblock Fine boundary regularity for the degenerate fractional $p$-{L}aplacian.
\newblock {\em Journal of Functional Analysis}, 279(8):1--54, 2020.

\bibitem[IMS15]{IannihS}
A.~Iannizzotto, S.~Mosconi, and M.~Squassina.
\newblock {$H^s$} versus {$C^0$}-weighted minimizers.
\newblock {\em NoDEA Nonlinear Differential Equations Appl.}, 22(3):477--497, 2015.

\bibitem[IMS16]{global}
A.~Iannizzotto, S.~Mosconi, and M.~Squassina.
\newblock Global hölder regularity for the fractional $p$-{L}aplacian.
\newblock {\em Rev. Mat. Iberoam.}, 32(4):1353--1392, 2016.

\bibitem[LL01]{Lieb}
E.~Lieb and M.~Loss.
\newblock {\em Analysis}, volume~14 of {\em Graduate Studies in Mathematics}.
\newblock American Mathematical Society, 2001.

\bibitem[LL14]{LindLind1}
E.~Lindgren and P.~Lindqvist.
\newblock Fractional eigenvalues.
\newblock {\em Calc. Var. Partial Differential Equations}, 49(1-2):795--826, 2014.

\bibitem[LLV23]{Lopera1}
E.~Lopera, C.~L\'{o}pez, and R.~E. Vidal.
\newblock Existence of positive solutions for a parameter fractional {$p$}-{L}aplacian problem with semipositone nonlinearity.
\newblock {\em J. Math. Anal. Appl.}, 526(2):Paper No. 127350, 12, 2023.

\bibitem[MMP14]{MMP}
D.~Montreanu, V.~V. Montreanu, and N.~Papageorgiou.
\newblock {\em Topological and Variational Methods with Applications to Nonlinear Boundary Value Problems}.
\newblock Springer, first edition, 2014.

\bibitem[MPSY16]{PerYang}
S.~Mosconi, K.~Perera, M.~Squassina, and Y.~Yang.
\newblock The {B}rezis-{N}irenberg problem for the fractional {$p$}-{L}aplacian.
\newblock {\em Calc. Var. Partial Differential Equations}, 55(4):Art. 105, 25, 2016.

\bibitem[MW89]{MW1}
J.~Mawhin and M.~Wilhem.
\newblock {\em Critical Point Theory and Hamiltonian Systems}.
\newblock Number~74 in Applied Mathematical Sciences. Springer--Verlag, 1989.

\bibitem[PS13]{PerSch}
K.~Perera and M.~Schechter.
\newblock {\em Topics in {C}ritical {P}oint {T}heory}.
\newblock Cambridge University Press, first edition, 2013.

\bibitem[Roy88]{Royden}
H.~Royden.
\newblock {\em Real Analysis}.
\newblock Prentice Hall, third edition, 1988.

\bibitem[RR20]{RecRumb5}
L.~Rec\^{o}va and A.~Rumbos.
\newblock Multiple nontrivial solutions of a semilinear elliptic problem with asymmetric nonlinearity.
\newblock {\em Journal of Mathematical Analysis and Applications}, 484/2, 123720:1--12, 2020.

\bibitem[SV12]{Ser1}
R.~Servadei and E.~Valdinoci.
\newblock Mountain pass solutions for non-local elliptic operators.
\newblock {\em J. Math. Anal. Appl.}, 389(2):887--898, 2012.

\bibitem[Val09]{Vald1}
E.~Valdinoci.
\newblock From the long jump random walk to the fractional {L}aplacian.
\newblock {\em Bol. Soc. Esp. Mat. Apl. SeMA}, 49:33--44, 2009.

\bibitem[Wan91]{Wang}
Z.~Q. Wang.
\newblock On a superlinear elliptic equation.
\newblock {\em Annales de l'{IHP}, {S}ection {C}}, 8:43--57, 1991.

\end{thebibliography}

\end{document}